\documentclass[preprint,3p]{elsarticle}

\usepackage[utf8]{inputenc}
\usepackage{tikz-cd} 
\usepackage{amsmath}
\usepackage{amsthm}
\usepackage{amsfonts}
\usepackage{subfig}
\usepackage{mathdots} 

 \usepackage{amssymb}
 \usepackage{amsthm}
 \usepackage{amsmath}
\usepackage{natbib}
 \usepackage{tikz}
 \usepackage{graphicx}
 \usepackage{epstopdf}
 \usepackage{xurl}
 \usepackage{hyperref}
 \usepackage{multirow}
 \usetikzlibrary{patterns}
\usepackage{comment}

\newtheorem{theorem}{Theorem}[section]

\newtheorem{lemma}[theorem]{Lemma}
\newtheorem{proposition}[theorem]{Proposition}
\newtheorem{corollary}[theorem]{Corollary}
\newtheorem{example}[theorem]{Example}

\newcommand{\Z}{\mathbb{Z}}

\definecolor{myYellow}{rgb}{1.0, 0.87, 0.0}
\definecolor{myYtable}{rgb}{1.0, 0.75, 0.0}
\definecolor{umass}{rgb}{0.004, 0.3215, 0.5568}
\definecolor{etonblue}{rgb}{0.59, 0.78, 0.64}
 \definecolor{mycolor}{rgb}{0.12, 0.3, 0.17}
\definecolor{myblue}{rgb}{0.03, 0.27, 0.49}
\definecolor{mediumjunglegreen}{rgb}{0.11, 0.21, 0.18}
\definecolor{asparagus}{rgb}{0.53, 0.66, 0.42}
\definecolor{goldenpoppy}{rgb}{0.99, 0.76, 0.0}
\definecolor{lincolngreen}{rgb}{0.11, 0.35, 0.02}
\definecolor{red(ncs)}{rgb}{0.77, 0.01, 0.2}
\definecolor{darkorange}{rgb}{1.0, 0.55, 0.0}
\definecolor{safetyorange(blazeorange)}{rgb}{1.0, 0.4, 0.0}
\definecolor{smokeytopaz}{rgb}{0.58, 0.25, 0.03}
\definecolor{brightgreen}{rgb}{0.4, 1.0, 0.0}
\definecolor{tue}{rgb}{0.82, 0.14, 0.14}

\newcommand{\ep}{\textcolor{black}}
\newcommand{\fr}{\textcolor{black}}
\newcommand{\rf}{\textcolor{black}}

\journal{arXiv}

\usepackage{xcolor}

\usepackage{xr}
\usepackage{heuristica}
\usepackage[heuristica,vvarbb,bigdelims]{newtxmath}
\usepackage{url}

\begin{document}

\begin{frontmatter}


	\title{Zero patterns in multi-way binary contingency tables with uniform margins}
	%
	%
    \author[PT]{Roberto Fontana}
    \author[TUe]{Elisa Perrone}
    \author[UG]{Fabio Rapallo}
	%
	%

    \address[PT]{
		Department of Mathematical Sciences, Politecnico di Torino, Corso Duca degli Abruzzi 24, 10129 Torino, Italy}
    \address[TUe]{Department of Mathematics and Computer Science, Eindhoven University of Technology, Groene Loper 3, 5612 AE Eindhoven, The Netherlands}
    \address[UG]{
        Department of Economics, Universit\`{a} di Genova, via Vivaldi 5, 16126 Genova, Italy}
\begin{abstract}
We study the problem of transforming a multi-way contingency table into an equivalent table with uniform margins and same dependence structure. This is an old question which relates to recent advances in copula modeling for discrete random vectors. In this work, we focus on multi-way binary tables and develop novel theory to show how the zero patterns affect the existence of the transformation as well as its statistical interpretability in terms of dependence structure. \ep{The implementation of the theory relies on combinatorial and linear programming techniques, which can also be applied to arbitrary multi-way tables. In addition, we investigate which odds ratios characterize the unique solution in relation to specific zero patterns.} Several examples are described to illustrate the approach and point to interesting future research directions.
\end{abstract}

\end{frontmatter}

\section{Introduction}
\label{Sec:Intro}
Largely employed across numerous domains such as healthcare, biology, and social sciences, contingency tables serve to display data in tabular format.
Contingency tables have been extensively analyzed within the field of statistics, primarily with the objective of developing methods to understand the dependence between variables (see, for example, \cite{rudas2018lectures}).
Recent work presented in \cite{geenens2020copula} highlights fascinating connections between the analysis of two-way contingency tables and copulas.
The essence of copula theory is centered on distinguishing between the influence of separate variables and their mutual dependency within the model. 
This separation facilitates the use of \emph{ad hoc} dependence modeling techniques by transforming the initial joint probability distribution into one with uniform margins on $[0,1]$, known as the copula. When the marginal distributions are continuous, the copula associated with the original distribution is uniquely defined (\cite{sklar_59}).
For discrete random variables, it is not possible to transform marginal distributions into uniform distributions using the Probability Integral Transform (PIT). The PIT identifies the copula solely within a specific subdomain. As a result, numerous copula models align with this specific subdomain.
The adaptation of copula theory to utilize an analogous notion for establishing a margin-free model within a discrete framework has been investigated in \cite{geenens2020copula,kojadinovic2024}.
In that study and the referenced literature, the authors explore the concept of transforming a given two-way contingency table into a new one that has uniform margins. 
This type of transformation clarifies the underlying relationships in the table, which could be obscured by notably uneven margins, as explained in the example below.

\begin{table}[b!]
    \centering
\normalsize
    \begin{tabular}{cc|c|c}
 \hline
    Vaccination  ($X_1$) & Recovery ($X_2$)& \quad $\Tilde{n}$ \quad  & \quad  $\Tilde{p}$ \quad \\
    \hline
no & no & 274 & 0.058 \\
no & yes & 278 & 0.059 \\
yes & no & 200 & 0.043 \\
yes & yes & 3951 & 0.840 \\
\hline \\
    \end{tabular}
    \caption{Sheffield smallpox epidemic reported in \cite{yule1912methods}.  \rf{The values of $\Tilde{p}$ are rounded to 3 decimal places.}}
    \label{tab:id7}
\end{table}

\smallskip
\noindent \textbf{{Motivating example.}} For $2 \times 2$ tables, the process entails converting the original table into a new one in which each marginal probability equals $1/2$, indicating a uniform distribution across the subdomain.
\rf{In Table \ref{tab:id7}}, we provide a classic instance from \cite{yule1912methods} that demonstrates smallpox cases documented at Sheffield Hospital. The dataset categorizes patients based on their vaccination status (yes or no) and their recovery outcome (yes or no).
The odds ratio, frequently used to evaluate associations in contingency table analysis, is notable with a value of $19.47$. Despite this, the unevenness of the marginals might obscure the link within the original dataset. The connection becomes apparent once the data is altered to show marginal probabilities of $1/2$, as illustrated in Fig.~\ref{fig:sheffield_tikz}~(b), while maintaining an odds ratio of~$19.47$.
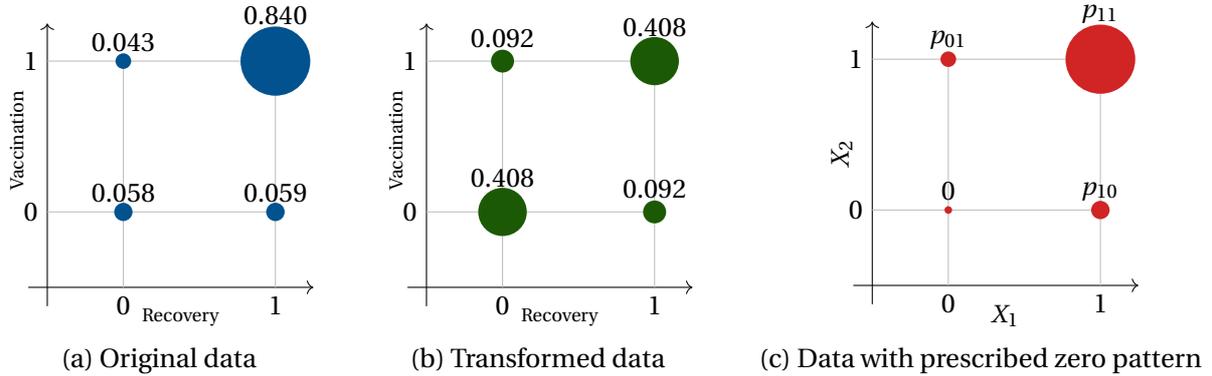
\begin{figure}[t!]
\centering
\begin{tabular}{ccccc}
\begin{tikzpicture}[scale=0.5]
\def\gridstep{2} 
\def\bubbleopacity{0.4} 
\def\bubblecolor{umass} 

\draw[->] (-0.5,0) -- (7,0);
\draw[->] (0,-0.5) -- (0,7);
\draw[gray!50, thin] (2,0) -- (2,6); 
\draw[gray!50, thin] (6,0) -- (6,6); 
\draw[gray!50, thin] (0,2) -- (6,2); 
\draw[gray!50, thin] (0,6) -- (6,6); 
\node[below] at (\gridstep,0) {0};
\node[below] at (3*\gridstep,0) {1};
\node[left] at (0,\gridstep) {0};
\node[left] at (0,3*\gridstep) {1};

\fill[\bubblecolor, opacity=\bubbleopacity] (1*\gridstep,1*\gridstep) circle[radius=0.241];
\node at (1*\gridstep,1*\gridstep+0.5) {0.058};

\fill[\bubblecolor, opacity=\bubbleopacity] (3*\gridstep,1*\gridstep) circle[radius=0.243];
\node at (3*\gridstep,1*\gridstep+0.5) {0.059};

\fill[\bubblecolor, opacity=\bubbleopacity] (1*\gridstep,3*\gridstep) circle[radius=0.206];
\node at (1*\gridstep,3*\gridstep+0.5) {0.043};

\fill[\bubblecolor, opacity=\bubbleopacity] (3*\gridstep,3*\gridstep) circle[radius=0.917];
\node at (3*\gridstep,3*\gridstep+1.2) {0.840};

\node at (3.5,-0.8) {\scriptsize{Recovery}};
\node[rotate=90] at (-0.8,3.9) {\scriptsize{Vaccination}};

\end{tikzpicture} & \qquad &
\begin{tikzpicture}[scale=0.5]
\def\gridstep{2} 
\def\bubbleopacity{0.4} 
\def\bubblecolor{lincolngreen} 

\draw[->] (-0.5,0) -- (7,0);
\draw[->] (0,-0.5) -- (0,7);
\draw[gray!50, thin] (2,0) -- (2,6); 
\draw[gray!50, thin] (6,0) -- (6,6); 
\draw[gray!50, thin] (0,2) -- (6,2); 
\draw[gray!50, thin] (0,6) -- (6,6); 
\node[below] at (\gridstep,0) {0};
\node[below] at (3*\gridstep,0) {1};
\node[left] at (0,\gridstep) {0};
\node[left] at (0,3*\gridstep) {1};

\fill[\bubblecolor, opacity=\bubbleopacity] (1*\gridstep,1*\gridstep) circle[radius=0.639];
\node at (1*\gridstep,1*\gridstep+0.9) {0.408};

\fill[\bubblecolor, opacity=\bubbleopacity] (3*\gridstep,1*\gridstep) circle[radius=0.303];
\node at (3*\gridstep,1*\gridstep+0.6) {0.092};

\fill[\bubblecolor, opacity=\bubbleopacity] (1*\gridstep,3*\gridstep) circle[radius=0.303];
\node at (1*\gridstep,3*\gridstep+0.6) {0.092};

\fill[\bubblecolor, opacity=\bubbleopacity] (3*\gridstep,3*\gridstep) circle[radius=0.639];
\node at (3*\gridstep,3*\gridstep+0.9) {0.408};

\node at (3.5,-0.8) {\scriptsize{Recovery}};
\node[rotate=90] at (-0.8,3.9) {\scriptsize{Vaccination}};

\end{tikzpicture} & \quad & \begin{tikzpicture}[scale=0.5]
\def\gridstep{2} 
\def\bubbleopacity{0.4} 
\def\bubblecolor{tue} 

\draw[->] (-0.5,0) -- (7,0);
\draw[->] (0,-0.5) -- (0,7);
\draw[gray!50, thin] (2,0) -- (2,6); 
\draw[gray!50, thin] (6,0) -- (6,6); 
\draw[gray!50, thin] (0,2) -- (6,2); 
\draw[gray!50, thin] (0,6) -- (6,6); 
\node[below] at (\gridstep,0) {0};
\node[below] at (3*\gridstep,0) {1};
\node[left] at (0,\gridstep) {0};
\node[left] at (0,3*\gridstep) {1};

\fill[\bubblecolor, opacity=\bubbleopacity] (1*\gridstep,1*\gridstep) circle[radius=0.1];
\node at (1*\gridstep,1*\gridstep+0.5) {0};

\fill[\bubblecolor, opacity=\bubbleopacity] (3*\gridstep,1*\gridstep) circle[radius=0.243];
\node at (3*\gridstep,1*\gridstep+0.5) {$p_{10}$};

\fill[\bubblecolor, opacity=\bubbleopacity] (1*\gridstep,3*\gridstep) circle[radius=0.206];
\node at (1*\gridstep,3*\gridstep+0.5) {$p_{01}$};

\fill[\bubblecolor, opacity=\bubbleopacity] (3*\gridstep,3*\gridstep) circle[radius=0.917];
\node at (3*\gridstep,3*\gridstep+1.2) {$p_{11}$};

\node at (3.5,-0.8) {$X_1$};
\node[rotate=90] at (-0.8,3.5) {$X_2$};

\end{tikzpicture}
\\
(a) Original data  & \qquad & (b) Transformed data & \qquad & (c) Data with prescribed zero pattern\\
\end{tabular}
\caption{Bubble plot of original data (a) and transformed data (b) for the Sheffield smallpox epidemic data in Table~\ref{tab:id7}. Part (c) represents a possible zero pattern in the observed table that motivates the investigations in this paper.}
\label{fig:sheffield_tikz}
\end{figure}

\ep{The transformed table also has an interpretation in terms of the information projection (I-projection) of the original probability table on the space of tables with uniform margins and prescribed support (see, e.g., \cite{geenens2020copula, kojadinovic2024}). In fact, this is the closest element to the original table with respect to the Kullback-Leibler divergence, and as such it is unique \cite{Csiszar_1975}. 
The uniqueness of the transformation is guaranteed only if a table with fixed margins and zero-pattern exists. However, this might not be the case when the original table contains zero entries.}
To clarify this, we consider a modification of the above example where the probability mass $p_{00}$ is zero. A generic table with this structure is represented in Fig.~\ref{fig:sheffield_tikz}~(c). 
It is straightforward to check that a table that preserves such a zero-structure and has uniform margins does not exist. Indeed, if such a table would exist, it would be the solution of the following system of equations
\begin{equation}
\left\{
\begin{aligned}
p_{00} +p_{01} +p_{10} +p_{11} &=1\\
p_{00} +p_{01} - p_{10}  - p_{11} &=0\\
p_{00} - p_{01} +p_{10} - p_{11} &=0 \\
p_{00} &= 0
\end{aligned}
\right. .
\label{eq:1}
\end{equation}
However, the system in Eq.~\eqref{eq:1} has no solutions unless $p_{11}$ is also set to zero, which means the original zero pattern is not maintained. 

\smallskip
\noindent \textbf{{Research problem and paper's contribution.}} The situation described in Fig.~\ref{fig:sheffield_tikz}~(c) is not artificial as zero entries are often observed in real datasets. 
Being an issue in applications, the problem of zero entries in contingency table analysis has caught the attention of researchers, especially in Algebraic Statistics, see for instance \cite[Chap.~9]{sullivant2018}. 
When working with sampling distributions, zeros can appear both as structural zeros and sampling zeros, a far more common situation than the structural zeros: As the number of cells increases in multi-way tables, so does the probability of sampling zeros even when the sample size is moderate or large.
In both cases, some of the \ep{conditional} odds ratios become zero or undefined, leading to the removal of the corresponding odds ratio equation in the transformation.
As a consequence, \ep{conditional} odds ratios become uninformative of the relationship between the variables and \ep{new constraints on higher order combinations of the conditional odds ratios must be added to characterize the transformed table that preserves the original odds ratio structure.} 
%

In this work, we further develop the initial investigations of \cite{Perrone2024, Fontana2023} to \ep{address the case of zeros in binary multi-way contingency tables.}
Our \emph{problem statement} is the following: Given a table $T$ (a probability table or a contingency table of observed frequencies), find a table $\tilde T$ with the uniform margins and the same odds ratios structure as $T$.
\ep{It is known that such a table $\tilde T$ is unique and can be find through scaling algorithms if and only if there exists a table with same zero pattern of $T$ and uniform margins (e.g., \cite{Csiszar_1975}). However, a characterization of the existence of $\tilde{T}$ in terms of possible zero patterns is not available in dimension $d>2$. In this paper, we provide conditions to ensure the existence of such a $\tilde T$.}
\ep{Under existence, we further investigate which combinations of conditional odds ratios are required to derive the unique solution in relation to the corresponding zero pattern.} 

The paper is organized as follows: In Section~\ref{Sec:background}, we describe the mathematical context of this work, introduce the notation, and connect our research to the state-of-the-art. 
Section~\ref{Sec:zero-existence} presents novel theoretical results on the existence of a transformed table with certain zero patterns. 
In Section~\ref{Sec:zero-uniqueness}, we show that in some degenerate cases, there might be a wider class of tables that satisfy the \ep{conditional} odds ratio constraints and have uniform margins. \ep{Therefore, higher order combinations of the primary conditional odds ratios are needed to obtain the unique transformation.}
We conclude the paper with a discussion in Section~\ref{Sec:discussion}.

\section{Background, notation, and relevance to the state-of-the-art}
\label{Sec:background}

In this work, we consider the case of $d$-dimensional multi-way binary contingency tables with zero entries and study how the presence of zeros affects the possibility of transforming a given table into one with uniform margins. 
In particular, we analyze all possible zero patterns where a zero pattern represents the set of the table cells whose corresponding frequencies are zero. 

Here, we consider $d$ binary factors, $X_1,\ldots,X_d$, i.e., a $d$-dimensional random vector $(X_1,\ldots,X_d)$. 
We refer to a specific cell of a table $T$ with the binary vector $\alpha=(\alpha_1,\ldots,\alpha_d) \in \{0,1\}^d$. 
We assume the set $\{0,1\}^d$ in lexicographic order and then we can refer to the $k$-th cell of $T$, $k=1,\ldots, 2^d$. We observe that the $k$-th cell of a table is the one corresponding to \rf{$\alpha_k=((\alpha_{k})_1, \ldots, (\alpha_{k})_d)$ such that $\sum_{j=1}^{d} (\alpha_{k})_j 2^{d-j}=k-1, \; k=1,\ldots,2^d$}, i.e., $\alpha_k$ is the binary representation of $k-1$.
For each cell $\alpha$, $\alpha \in \{0,1\}^d$, we define $p_\alpha$, the relative frequency of the cell. Thus, a table $T$ is described by the $2^d$-vector $p=(p_\alpha, \alpha \in \{0,1\}^d)$.
In the $d$-dimensional case, there are \rf{$2^{2^d}$} possible zero patterns, including the $2^d$-zero-pattern that results in a trivial all-zero table.  We can represent each zero pattern $\mathcal{Z}$ with a binary $2^d$ vector, i.e., $\mathcal{Z}=(z_1,\ldots,z_{2^d})$, where $z_i=0$ ($z_i=1$) indicates that the corresponding cell $\alpha_{i}$  has $p_{\alpha_{i}}=0$ ($p_{\alpha_{i}}>0$, respectively). 
We write that a zero pattern $\mathcal{Z}$ is a $k$-zero pattern if the number of zeros in $\mathcal{Z}$ is $k$, with $k=0,\ldots,2^d$.
Finally, we say that a table $T$ with uniform margin is $\mathcal{Z}$-compatible if the corresponding probability mass function (pmf) $p$ satisfies that $p_{\alpha_{i}}=0$ ($p_{\alpha_{i}}>0$) when $z_i=0$ ($z_i=1$, respectively), $i=1,\ldots,2^d$.

We now introduce the notion of conditional odds ratios, which is key to this work. In particular, we define the conditional odds ratios constraints on each $2 \times 2$ sub-tables by fixing $d-2$ variables. When $d=2$, there is only one odds ratio given by
\[
\omega^{12} =\frac{{p_{11}}p_{00}}{{p_{10}}p_{01}}.
\]
In the case $d=3$ there are six \ep{conditional} odds ratios given by
\begin{equation} \label{orsys}
\begin{split}
\omega_{0}^{23}= \frac {p_{000}p_{011}} {{p_{001}p_{010}}} \qquad \omega_{1}^{23}=\frac {p_{100}p_{111}} {{p_{101}p_{110}}}
\\
\omega_{0}^{13}=\frac {p_{000}p_{101}} {{p_{001}p_{100}}} \qquad \omega_{1}^{13}= \frac {p_{010}p_{111}} {{p_{011}p_{110}}}
\\
\omega_{0}^{12}=\frac {p_{000}p_{110}} {{p_{010}p_{100}}} \qquad 
\omega_{1}^{12}=\frac {p_{001}p_{111}} {{p_{011}p_{101}}}
\end{split} 
\end{equation}
The notation $\omega_{0}^{23}$ denotes the \ep{conditional} odds ratio for the variables $X_2$ and $X_3$ given the value $0$ for $X_1$. The conditions in Eq.~\eqref{orsys} are not independent, and only four of them are independent.

In arbitrary dimension $d>3$, the generic conditional odds ratio is defined in the same way: for each pair of variables $X_i$ and $X_j$ we consider the conditional odds ratio for given values of the other $(d-2)$ variables. If we denote with $\alpha'\in \{0,1\}^{d-2}$ the values of the fixed $(d-2)$ variables, the \ep{conditional} odds ratio is denoted as $\omega_{\alpha'}^{ij}$ and defined by: 
\begin{equation} \label{oreq}
\omega_{\alpha'}^{ij} =\frac{{p_{\alpha_1}}p_{\alpha_2}}{{p_{\alpha_3}}p_{\alpha_4}},
\end{equation}
where $\alpha_1, \alpha_2, \alpha_3, \alpha_4$ are equal to $\alpha'\in \{0,1\}^{d-2}$ in the entries $\{1, \ldots, d\} - \{i,j\}$ and are equal to $(1,1),(0,0),$ $(1,0),(0,1)$ respectively in the entries $(i,j)$. 
\rf{An example in $d=4$ is
\[
\omega_{01}^{13} =\frac{{p_{0001}}p_{1011}}{{p_{0011}}p_{1001}}.
\]}
In the $d$-way case there are $\binom{d}{2}2^{d-2}$ equations, which are not independent. \ep{In the remainder of the paper, we refer to the conditional odds ratios defined in Eq.~(\ref{oreq}) simply as \textit{odds ratios}.}
In the following, we draw connections with previous research on discrete copulas and classical results for log-linear models.

\subsection{Table transformations and discrete copulas}
In this section, we related our research problems to the theory of discrete copulas. For the sake of clarity, we first restrict to the two-dimensional case.
We consider $R\in \Z_{>0}$ and denote $I_R=\{ 0, 1/R, \ldots, (R-1)/R, 1 \}$, $[R] = \{1,\ldots,R\}$, and $\langle R\rangle = \{0,\ldots,R\}$.
Given $R$ and $S$ as positive integers, we define $U_R=\{u_0=0, u_1, \ldots, u_{R-1}, u_R=1\}$, with $u_0 < \ldots < u_R$, and $V_S=\{v_0=0, v_1, \ldots, v_{S-1}, v_S=1\}$, with $v_0 < \ldots < v_S$, as two discrete grid partitions over the unit interval. A discrete copula $C_{U_R,V_S}$ is defined on the set $U_R \times V_S$ and retains the characteristic properties of a copula function over the grid domain $U_R \times V_S$.
As noted in \cite{Perrone2021,Perrone2022}, interesting relationships exist between the domain of discrete copulas and certain convex polytopes known as \emph{transportation polytopes}. 
These polytopes are also associated with the contingency table analysis and the existence of a table with prescribed margins \citep{DeLoera_14}.
Namely, considering two vectors $\tilde{u}=(\tilde{u}_1, \ldots, \tilde{u}_R) \in \mathbb{R}^R_{>0}$ and $\tilde{v}=(\tilde{v}_1, \ldots, \tilde{v}_S) \in \mathbb{R}^S_{>0}$, we can define the transportation polytope $\mathcal{T}(\tilde{u},\tilde{v})$ as the convex polytope in the $R S$ variables $x_{i,j}$ satisfying,  for all $i\in[R]$ and $j\in[S]$, the following conditions:
$
x_{i,j}\geq0,
\quad
\sum_{h=1}^Sx_{i,h} = \tilde{u}_i,
\quad
\sum_{\ell=1}^Rx_{\ell,j} = \tilde{v}_j.
$
The two vectors $\tilde{u}$ and $\tilde{v}$ are called the margins of $\mathcal{T}(\tilde{u},\tilde{v})$. 
According to \cite{Perrone_2019}, each discrete copula $C_{U_R,V_S}$ is associated with a matrix in a transportation polytope $\mathcal{T}(\tilde{u},\tilde{v})$, and conversely.
The transportation matrix is connected to the probability mass function of the discrete random vector, whereas the associated discrete copula corresponds to the cumulative distribution function.

We now show how to derive the discrete copula associated with a given contingency table.
We consider the motivating example presented in Table~\ref{tab:id7}. There, we get $N=4703$ total observations. In this example, $R=S=2$, the vectors $\tilde{u}$ and $\tilde{v}$ are the margins of the contingency table, i.e., $\tilde{u}=(552, 4151)$ and $\tilde{v}=(474,4229)$, while the defining grids of the corresponding discrete copula are $U_2=\frac{1}{N}\{0,\tilde{u}_1,\tilde{u}_1+\tilde{u}_2\}=\{0, 0.12,1\}$ and $V_2=\{0,0.1,1\}$.
The entries of the discrete copula $C_1 = C_{U_2,V_2}=(c_{i,j})$, $i \in [2]$ and $j \in [2]$ are computed from the entries of the contingency table $(x_{i,j})$ by summing up and normalizing, i.e., $c_{i,j} = \frac{1}{N} \sum_{\ell =1}^i \sum_{h=1}^j x_{\ell,h}$, while $c_{0,0}=c_{i,0}=c_{0,j}=0$, for $i \in [2]$ and $j \in [2]$. Namely, $C_1$ is as follows:
\[
C_1=\begin{pmatrix} \begin{array}{ccccc}
                       0 & 0.00 & 0.00  \\
                        0 & 0.058 &  0.12 \\
                        0 & 0.1 &  1.00
                    \end{array} \end{pmatrix} \, .
\]
As discussed in the introduction, in \cite{geenens2020copula}, the author emphasizes the challenge of inferring dependence from tables with non-uniform margins, such as Table~\ref{tab:id7}.
In line with copula theory for continuous random variables, the author proposes finding a representative $\bar{\mathbf{p}}$ for all $(R \times S)$ probability distributions that (1) maintains the inter-dependencies in a contingency table using odds ratios, and (2) has uniform margins of $1/R$ and $1/S$. 
In discrete copulas terms, this corresponds to finding a discrete copula defined on the rectangular grid $I_R \times I_S$ that maintains the dependence structure of the initial discrete copula derived from a specific contingency table.
In the example above, the problem is to identify a discrete copula $\tilde{C}_1$ with its domain on $I_2 \times I_2$ and margins $I_2$, which is somehow compatible with the original discrete copula $C_1$. 

A natural question that arises is whether or not such an element exists. \ep{If it exists, it is unique in line with the theory of I-projections \cite{Csiszar_1975}.}
For two-dimensional tables, the answer to this question is provided by a theorem presented in \cite{geenens2020copula}, which we summarize below in a simplified version. \rf{The theorem examines the structure of the given zero pattern in terms of the \emph{rectangular sets} it contains, establishing conditions based on their ``size''.}

\begin{theorem}
\label{th:Geenens_6.1}
Let $\mathbf{p}$ be in the set $\mathcal{P}_{R \times S}$ of all $(R \times S)$ probability distributions.  We define $\text{Supp}(\mathbf{p})=\{(i,j) \in [R] \times [S] \text{ s.t. } p_{i,j}>0\}$ \rf{, the support of $\mathbf{p}$,} and 
$\text{N}(\mathbf{p})=\{(v_{X_1} \times v_{X_2}): v_{X_1} \subset [R], v_{X_2} \subset [S] s.t. \sum\limits_{(i,j) \in v_{X_1} \times v_{X_2}} p_{i,j}=0\}$, the set of rectangular subset\rf{s} of $[R] \times [S]$ where $\mathbf{p}$ is null. The cardinality of a set $A$ is denoted by $|A|$.
\begin{enumerate}
\item Suppose that for all $(v_{X_1} \times v_{X_2}) \in N(\mathbf{p})$, $\frac{|v_{X_1}|}{R} + \frac{|v_{X_2}|}{S} < 1$, then there exists a unique $\bar{\mathbf{p}}$ with uniform margins associated with a discrete copula $C_{I_R \times I_S}$.
\item Suppose that for all $(v_{X_1} \times v_{X_2}) \in N(\mathbf{p})$, $\frac{|v_{X_1}|}{R} + \frac{|v_{X_2}|}{S} \leq 1$ with $\frac{|\tilde{v}_{X_1}|}{R} + \frac{|\tilde{v}_{X_2}|}{S}= 1$ for some $(\tilde{v}_{X_1} \times \tilde{v}_{X_2}) \in N(\mathbf{p})$.
\begin{itemize}
\item[(i)] If, for all $(\tilde{v}_{X_1} \times \tilde{v}_{X_2}) \in N(\mathbf{p})$ such that $\frac{|\tilde{v}_{X_1}|}{R} + \frac{|\tilde{v}_{X_2}|}{S} = 1, ([R] \setminus \tilde{v}_{X_1} \times [S] \setminus \tilde{v}_{X_2}) \in N(\mathbf{p})$, then there exists a unique $\bar{\mathbf{p}}$ with uniform margins associated with a discrete copula $C_{I_R \times I_S}$.
\item[(ii)] If, there exists $(\tilde{v}^*_{X_1} \times \tilde{v}^*_{X_2}) \in N(\mathbf{p})$ such that $\frac{|\tilde{v}^*_{X_1}|}{R} + \frac{|\tilde{v}^*_{X_2}|}{S} = 1, and ([R] \setminus \tilde{v}^*_{X_1} \times [S] \setminus \tilde{v}^*_{X_2}) \notin N(\mathbf{p})$, then there is no element $\bar{\mathbf{p}}$ with uniform margins and same support of the original table, but there is one element $\bar{\mathbf{p}}$ with support \rf{strictly contained in that of $\mathbf{p}$,  $\text{Supp}(\bar{\mathbf{p}}) \subset   \text{Supp}(\mathbf{p})$, and uniform margins}.
\end{itemize}
\item Suppose that there exists $\tilde{\tilde{v}}_{X_1} \times \tilde{\tilde{v}}_{X_2} \in N(\mathbf{p})$ such that $\frac{|\tilde{\tilde{v}}_{X_1}|}{R} + \frac{|\tilde{\tilde{v}}_{X_2}|}{S}> 1$. Then there is no element $\bar{\mathbf{p}}$ with uniform margins such that it has same odds ratio structure of $\mathbf{p}$ and is associated with a discrete copula $C_{I_R \times I_S}$, not even with a modified support.
\end{enumerate}
\end{theorem}
We can interpret the example of Table~\ref{tab:id7} in light of Thm.~\ref{th:Geenens_6.1}. For a broader discussion of the two-dimensional binary case, we refer the reader to \cite{geenens2020copula}. 
The case reported in Table~\ref{tab:id7} does not show any zero pattern, since all entries are strictly positive. Therefore, the transformed table exists and is unique in agreement with Part~1 of Thm.~\ref{th:Geenens_6.1}. 
We now analyze the situation depicted in Fig.~\ref{fig:sheffield_tikz}~(c) where one entry is zero while the others are non-zero. 
This case falls under point 2~(ii) of Thm.~\ref{th:Geenens_6.1}. Specifically, $|v_{X_1}| = |v_{X_2}| = 1$. Hence, $\frac{|\tilde{v}^*_{X_1}|}{2} + \frac{|\tilde{v}^*_{X_2}|}{2} = 1$ and a solution \ep{with exactly the same support} does not exist in agreement with the conclusion of the introduction. 
If we move to the boundary of the probability space by adding an extra zero in the zero pattern (that is, allowing $p_{11}$ to be set equal to zero), a solution exists. 
When it exists, the element $\bar{\mathbf{p}}$ can be obtained using the Iterative Proportional Fitting Procedure (IPFP), which is a standard method in contingency table analysis for a meaningful comparison of tables with different margins and same dependence structure in terms of corresponding odds ratios \ep{and their higher order combinations} \cite{rudas2018lectures,geenens2020copula}.
\rf{We consider an $R \times S$ table $T$ and a probability distribution $p=(p_{ij}:i=1,\ldots,R, \; j=1,\ldots S)$ defined over $T$. Without loss of generality, we assume all marginals $p_{i+}=\sum_{j=1}^{S} p_{ij}$ and $p_{+j}=\sum_{i=1}^{R} p_{ij}$ to be nonzero, $i=1,\ldots,R, \; j=1,\ldots S$. At the $k$-th iteration, $k=1,\ldots,N_{\max}$, the IPFP algorithm performs the following \emph{row} and \emph{column} transformations:
\begin{eqnarray}
p_{ij}^{(k+1)}=\frac{1/R}{p_{i+}^{(k)}} p_{ij}^{(k)}, \;\; i=1,\ldots,R, \; j=1,\ldots S  \\
p_{ij}^{(k+1)}=\frac{1/S}{p_{+j}^{(k+1)}} p_{ij}^{(k+1)},  \;\; i=1,\ldots,R, \; j=1,\ldots S,     
\end{eqnarray}  
where $p_{ij}^{(1)}=p_{ij}$ and $N_{\max}$ is a predefined maximum number of iterations. The output of the algorithm is a new table $T'$ with pmf $p_{ij}^{(N_{\max})}$ that, under the conditions of Thm.~\ref{th:Geenens_6.1}, has uniform margins. For further details, particularly regarding the extension to $d$-way tables, we refer the interested reader to \cite{barthelemy2018mipfp}.} \ep{
It should be noted that IPFP preserves the support and odds ratio structure in terms of the odds ratio matrix, i.e., the matrix of all conditional odds ratios, and the higher order interactions of the odds ratio of the initial table across the iterations. 
If there are no zeros in the support, the odds ratio matrix gives sufficient constraints to identify the transformed table. With an example, we now show that if zeros are present in the table, the conditional odds ratios alone do not determine the table, and higher order combinations of the odds ratios must be considered.} 
\fr{We examine the following $3 \times 3$ table with three zeros on the main diagonal:
\[
\begin{pmatrix} \begin{array}{ccc}
                       0 & p_{12} & p_{13}  \\
                        p_{21} & 0 &  p_{23} \\
                        p_{31} & p_{32} &  0
                    \end{array} \end{pmatrix} \, .
\]
This table falls within the existence case of Thm.~\ref{th:Geenens_6.1}. Nevertheless, the odds-ratios defined on $2 \times 2$ sub-tables are not enough to identify the equivalent table. In fact, no odds ratio on $2 \times 2$ sub-tables is well defined, and to identify the transformed table we need to fix also the quantity
\begin{equation} \label{sixpieces}
\frac {p_{12}p_{23}p_{31}} {p_{13}p_{21}p_{32}} \, ,
\end{equation}
which arises as the symbolic product
\begin{equation} \label{symbcanc}
\frac {p_{12}p_{23}} {p_{13}p_{22}} \cdot \frac {p_{22}p_{31}} {p_{21}p_{32}} \, .
\end{equation}
The quantity reported in Eq.~\eqref{sixpieces} is well known in Algebraic Statistics for tables with structural zeros: It arises from elimination theory, which provides a formal and rigorous justification for the symbolic cancelation in Eq.~\eqref{symbcanc}. For more details, see, e.g. \cite{rapallo06}. In Sec.~\ref{Sec:zero-uniqueness}, we analyze similar examples in dimension three. 
}

In the following, we further discuss how our problem relates to standard theory in log-linear models.

\subsection{Relations with the classical theory of log-linear models}

In log-linear models, there is a well-established link between IPFP and Maximum Likelihood Estimation (MLE). The use of IPFP to find MLE in complete tables dates back to \cite{Fienberg70} (see also \cite{FienbergMeyer2006}). IPFP works in log-linear models as follows: First, define a table satisfying all the desired independence and/or conditional independence statements, and then, run the IPFP to adjust the table with respect to the relevant marginal distributions of the sufficient statistic. 
When there are zeros in the table, the existence of the MLE is generally not guaranteed. 
Important developments in this case have been achieved with the use of Algebraic Statistics: In most cases there are criteria to identify zero patterns which preserve the existence of the MLE, see \cite{fienbergrinaldo12}. 
In large classes of statistical models, including hierarchical log-linear models and staged trees, recent results have been introduced in \cite{coons24}. Despite these connections, our problem is different because the margins are constrained to be uniform, while in log-linear models the margins are determined by suitable margins of the observed table. 
To clarify this key difference, we can refer again to the example depicted in Fig.~\ref{fig:sheffield_tikz}~(c). In that situation, it is clear that there is no table with uniform margins and the prescribed support. 
However, if we start with unbalanced margins, i.e., $(1/3,2/3)$, and relax the constraints of uniform margins, we can easily obtain tables that maintain the zero pattern.

\section{Impact of zeros in the existence of the transformed multi-way table}
\label{Sec:zero-existence}
As mentioned in the previous sections, the presence of zeros in the table affects the construction of a table with uniform margins in different ways. 
First, some odds ratio equations in Eq.~\eqref{oreq} become meaningless and can thus be removed from the system. 
Additionally, the presence of zeros reduces the support of the probability distribution, reducing the number of free probabilities in the table. 


As a prototype, let us consider again the $2 \times 2$ case.
Given a $2\times 2$ contingency table $(\Tilde{n}_{ij}, i,j=0,1)$ let $N=\Tilde{n}_{00}+\Tilde{n}_{11}+\Tilde{n}_{01}+\Tilde{n}_{10}$ be the grand total and $\omega^{12}=\frac{\Tilde{n}_{00}\Tilde{n}_{11}}{\Tilde{n}_{01}\Tilde{n}_{10}}$ the odds ratio. We examine the table in terms of the relative frequencies $\Tilde{p}_{ij}=\frac{\Tilde{n}_{ij}}{N}$ rather than the counts $\Tilde{n}_{ij}$.
As shown in the motivating example in Sect.~\ref{Sec:Intro}, the goal is to determine a new table $(p_{ij} \geq 0, i,j=0,1)$ such that the marginals are still $1/2$ and the original table's odds ratio $\omega^{12}$ is maintained. Such a table can be obtained by solving the following system of equations:

\begin{equation}
\left\{
\begin{aligned}
\frac{p_{00} p_{11}}{p_{01} p_{10}}&=\omega^{12} \\
p_{00} +p_{01} +p_{10} +p_{11} &=1\\
p_{00} +p_{01} - p_{10}  - p_{11} &=0\\
p_{00} - p_{01} +p_{10} - p_{11} &=0
\end{aligned}
\right. ,
\label{eq:d2}
\end{equation}
whose unique solution is given as follows:
\begin{equation}
p_{00}=p_{11}=\frac{\sqrt{\omega^{12}}}{2(1+\sqrt{\omega^{12}})}, \qquad 
p_{01}=p_{10}=\frac{1}{2(1+\sqrt{\omega^{12}})}.
\label{eq:sol2}
\end{equation}

In the described $2^2$ scenario, addressing the issue of zeros in the table is straightforward, with two possible situations: (a) The table is entirely filled, or (b) The table might contain one or more zero entries. 
As discussed in Sec.~\ref{Sec:background}, if the table is complete, there exists a unique table with identical odds ratios and uniform margins, obtainable through the IPFP. 
Referring again to \cite{geenens2020copula}, in scenario (b), odds ratios are absent, simplifying the task of identifying a table with the same support as the original and uniform margins. This scenario is only achievable when there are two zeros located at $(0,0), (1,1)$ or $(1,0),(0,1)$.
In an intermediate scenario with a single zero present in the table as in Fig.~\ref{fig:sheffield_tikz}~(c), the IPFP leads to a boundary solution, resulting in another cell probability approaching zero. In this case, there are no tables which are zero-compatible with the original table and also have uniform margins.
In the next section, we provide existence results for arbitrary dimensions and zero patterns.

\subsection{$d$-way binary tables} \label{sec:dwaytab}
We now move to the general $d$-way \ep{binary} case. We denote $\{0,1\}^k$ by $V_k$, $k=1,\ldots,d$. To simplify the notation, sometimes we write $\mathcal{D}$ instead of $V_d=\{0,1\}^d$.
Here, we are interested in tables with uniform margins. More specifically, there are $d$ margins $m_i, i=1,\ldots,d$ defined as
\[
m_i=(m_{i,0},m_{i,1})=\left(\sum_{\alpha \in \mathcal{D}, \alpha_i=0} p_\alpha, \sum_{\alpha \in \mathcal{D}, \alpha_i=1} p_\alpha \right) ,
\]
where $\mathcal{D}=\{0,1\}^d$.
We observe that $m_{i,1}=E[X_i]$ and $m_{i,0}+m_{i,1}=1$, $i=1,\ldots,d$.
To obtain uniform margins, the pmf of the table must satisfy the following constraints
\[
m_{i,0}=m_{i,1}=\frac{1}{2},  \; i=1\ldots,d\, .
\]
The next proposition \ref{prop:segmenti} states a general result on the structure of the tables with uniform margins that are compatible with a class of zero patterns. 

\begin{proposition} \label{prop:segmenti}
Let us consider a table $T$, with relative frequencies $p=(p_\alpha, \alpha \in V_d)$ and uniform margins. Let us suppose that \rf{, given $i_1,i_2 \in \{1,\ldots,d\}, i_1 \neq i_2,$ and $y_{i_1}, y_{i_2} \in \{0,1\} $}, $p_\alpha=0$ for $\alpha\in \{(\alpha_1,\ldots, \alpha_d) \in V_d: \alpha_{i_1}=y_{i_1}, \alpha_{i_2}=y_{i_2}\}$. Then $p_\alpha=0$ for $\alpha\in \{(\alpha_1,\ldots, \alpha_d) \in V_d: \alpha_{i_1}=1-y_{i_1}, \alpha_{i_2}=1-y_{i_2}\}$.
\end{proposition}
\begin{proof}
Without loss of generality, we can take $i_1=1$, $i_2=2$, $y_{i_1}=y_{i_2}=0$. The hypothesis on $p$ can be written as $p_{00\alpha}=0, \alpha \in V_{d-2}$.
We compute $m_{1,0}$, the first component of the first margin:
\[
m_{1,0}=\sum_{\alpha' \in V_{d-1}} p_{0\alpha'} = \sum_{\alpha'' \in V_{d-2}} p_{00\alpha''}+\sum_{\alpha'' \in V_{d-2}} p_{01\alpha''}= \sum_{\alpha'' \in V_{d-2}} p_{01\alpha''}\, .
\]
The last equality in the above equation follows by $p_{00\alpha}=0, \alpha \in V_{d-2}$. 
The table has uniform margins, in particular $m_{1,0}=\frac{1}{2}$, and as a consequence,
\[
\sum_{\alpha'' \in V_{d-2}} p_{01\alpha''}=\frac{1}{2}.
\]
Let us now consider the second component of the second margin, $m_{2,1}$. The table $T$ has uniform margins and then $m_{2,1}=\frac{1}{2}$. Thus, we obtain the following chain of equalities:
\[
m_{2,1}=\sum_{\alpha_1 \in \{0,1\}, \alpha'' \in V_{d-2}} p_{\alpha_1,1,\alpha''} = \sum_{\alpha'' \in V_{d-2}} p_{01\alpha''}+\sum_{\alpha'' \in V_{d-2}} p_{11\alpha''}= \frac{1}{2} +\sum_{\alpha'' \in V_{d-2}} p_{11\alpha''} \, .
\]
Consequently, we obtain $\sum_{\alpha'' \in V_{d-2}} p_{11\alpha''}=0$, that is $p_{11\alpha''}=0, \alpha'' \in V_{d-2}$.
\end{proof}

\rf{We observe that if the zero pattern is one of those expressed by the condition of Proposition \ref{prop:segmenti}, e.g., by choosing $i_1=1, i_2=2, y_{i_1}=0, y_{i_2}=0$ that means
\[
p_\alpha=\begin{cases}
0 & \alpha=(0,0,\alpha'') \text{ or } \alpha=(1,1,\alpha''), \; \alpha'' \in V_{d-2} \\
>0 & \text{ elsewhere} \\
\end{cases} \, .
\]
we can build a compatible table $T$ with uniform margins. 
}
Namely, we define
\[
p_\alpha=\begin{cases}
0 & \alpha=(0,0,\alpha'') \text{ or } \alpha=(1,1,\alpha''), \; \alpha'' \in V_{d-2} \\
\frac{1}{2^{d-1}} & \text{ elsewhere} \\
\end{cases} \, .
\]
Since all margins have a similar structure, it is enough to compute $m_{1,0}$ and $m_{3,0}$, which are:
\[
m_{1,0}=\sum_{\alpha'' \in V_{d-2}} p_{01\alpha''}=\sum_{\alpha'' \in V_{d-2}} \frac{1}{2^{d-1}}=\frac{2^{d-2}}{2^{d-1}}=\frac{1}{2},
\]
\[
m_{3,0}=\sum_{\alpha \in \mathcal {D}, \alpha_3=0} p_\alpha= \sum_{\alpha \in V_{d-3}} p_{010\alpha}+ \sum_{\alpha \in V_{d-3}} p_{100\alpha} = \frac{1}{2^{d-1}} 2^{d-3} + \frac{1}{2^{d-1}} 2^{d-3}= \frac{1}{2}.
\]
Finally, the condition $\sum_{\alpha \in \mathcal{D}} p_\alpha=1$ can be verified by observing, for example, that 
\[
\sum_{\alpha \in \mathcal{D}} p_\alpha =m_{1,0}+m_{1,1}=\frac{1}{2}+\frac{1}{2}=1.
\]

To further explore the existence of a table with uniform margins, we use the following lemma. The lemma is based on a rewriting of the multidimensional case Th.~3 in \cite{franklinlorenz}, and an adaptation of the notation to binary tables.

\begin{lemma} \label{simpllemma}
Let $T_1=(p_\alpha)$ and $T_2=(q_\alpha)$ be two $d$-dimensional binary probability tables with the same zero pattern, and assume that $T_2$ has uniform margins. Then there exist positive vectors $w_1=(w_{10},w_{11}), \ldots , w_d=(w_{d0},w_{d1})$ such that the table $T_3$ with generic element
\[
\tilde p_{\alpha_1 \cdots \alpha_d} = p_{\alpha_1 \cdots \alpha_d}w_{1\alpha_1} \cdots w_{d\alpha_d}
\]
has uniform margins.
\end{lemma}

\rf{First of all we point out that the transformation in Lemma \ref{simpllemma} preserves the zero pattern and the odds ratio structure of the table.} We now highlight two key points on the use of Lemma \ref{simpllemma}. First, for the existence of a solution to our problem, it is enough to check the conditions on margins, since the conditions on the odds ratios \ep{and their higher order combinations} are automatically adjusted by the IPFP. Second, it is enough to prove the existence of a table with the same zero pattern and uniform margins, and it is not even necessary to actually determine such a table.

%
To do this, we propose two methods. The first method is based on combinatorial objects, while the second method is based on integer programming.

We point out that to have uniform margins the pmf $p$ of a table $T$ must satisfy the following constraints
\[
m_{i0}=m_{i1}=\frac{1}{2} \Leftrightarrow m_{i0}-m_{i1}=0 \Leftrightarrow \sum_{\alpha \in \mathcal{D}, \alpha_i=0} p_\alpha- \sum_{\alpha \in \mathcal{D}, \alpha_i=1} p_\alpha=0,  \; i=1\ldots,d
\]
that can be rewritten as
\begin{equation} \label{eq:system}
 (1_{\{\alpha \in \mathcal{D}, \alpha_i=0\}}-1_{\{\alpha \in \mathcal{D}, \alpha_i=1\}})^Tp=0, \;\;i=1\ldots,d \, , 
\end{equation}

where $1_{A}$ is the indicator vector of $A$, $b^T$ denotes the transpose of $b$, and $p$ is the vector $p=(p_\alpha, \alpha \in \mathcal{D})$. Eq.~\eqref{eq:system} can be rewritten in compact form by denoting as $H_d$ its matrix of coefficients. Namely, Eq.~\eqref{eq:system} becomes:
\begin{equation} \label{eq:HD}
    H_d \,p=0.
\end{equation}

For example, for $d=3$ we obtain
\[
H_3 p=
\begin{pmatrix}
1 & 1 & 1 & 1 & -1 & -1 & -1 & -1  \\
1 & 1 & -1 & -1 & 1 & 1 & -1 & -1 \\
1 & -1 & 1 & -1 & 1 & -1 & 1 & -1
\end{pmatrix}p=0 \, .
\]

The extreme rays of the system defined by $H_d y = 0$ with $y=(y_\alpha, \alpha \in \mathcal{D}, y_\alpha \geq 0)$ can be computed using software for commutative algebra. We used 4ti2, \cite{4ti2}. Then the corresponding extreme pmfs are computed by simple normalization $p_\alpha=\frac{y_\alpha}{\sum_{\alpha \in \mathcal{D}} y_\alpha}$.
For the $3$-dimensional case we obtain the extreme pmfs that are shown in Table \ref{tab:ex3}.
\begin{table}[h]
	\centering
		\begin{tabular}{ccc|rrrrrr}
$\alpha_1$ &	$\alpha_2$ &	$\alpha_3$ &	$r_1$ & 		$r_2$ & 		$r_3$ & 		$r_4$ & 		$r_5$ & 		$r_6$ \\
		\hline
0 &	0 &	0 &	$\frac{1}{2}$ &	$0$ &	$0$ &	$0$ &	$\frac{1}{4}$&	$0$ \\
1 &	0 &	0 &	$0$ &	$\frac{1}{2}$ &	$0$ &	$0$ &	$0$ & $\frac{1}{4}$ \\
0 &	1 &	0 &	$0$ &	$0$  & $\frac{1}{2}$ &	$0$ &	$0$ & $\frac{1}{4}$ \\
1 &	1 &	0 &	$0$ &	$0$  & $0$  & $\frac{1}{2}$ & $\frac{1}{4}$ &	$0$  \\
0 &	0 &	1 &	$0$ &	$0$  & $0$  & $\frac{1}{2}$  &	$0$ & $\frac{1}{4}$ \\
1 &	0 &	1 &	$0$ &	$0$  & $\frac{1}{2}$ &	$0$ & $\frac{1}{4}$ & $0$  \\
0 &	1 &	1 &	$0$ &	$\frac{1}{2}$ &	$0$ &	$0$ & $\frac{1}{4}$ &	$0$  \\
1 &	1 &	1 &	$\frac{1}{2}$ &	$0$ &	$0$ &	$0$ &	$0$ & $\frac{1}{4}$ \\
		\end{tabular}
	\caption{Extreme pmfs with uniform margins $d=3$.}
	\label{tab:ex3}
\end{table}

It is worth noting that, given the extreme pmfs $r_i, i=1,\ldots,n_d$ where $n_d$ depends on the dimension $d$ of the multi-way table (e.g., for the three-dimensional case we have $n_3=6$) any pmfs $p$ with uniform margins can be written has $p=\sum_{i=1}^{n_d} \lambda_i r_i$ with $\lambda_i\geq0, i=1\ldots,n_d$ and $\sum_{i=1}^{n_d} \lambda_i=1$.
In Thm.~\ref{thm:extremerays}, we provide conditions that a zero pattern must satisfy to be compatible with the existence of a table with uniform margins. 

\begin{theorem}\label{thm:extremerays}
Let $\mathcal{Z}$ be a zero pattern with $\mathcal{Z}=(z_i, i=1,\ldots,2^d)$. Let us denote by $\mathcal{A}$ the set of indices corresponding to zero cells, $\mathcal{A}=\{i \in \{0,\ldots,2^d\}: z_i=0\}$ and $\bar{\mathcal{A}}$ the set of the indices corresponding to non-zero cells, $\bar{\mathcal{A}}=\{i \in \{0,\ldots,2^d\}: z_i=1\}$. Let $\{r_i: i=1,\ldots,n_d\}$ be the set of extreme pmfs of the polytope defined by $H_d \, p=0$.

Let $\mathcal{S}_1$ be the subset of $\{1,\ldots,n_d\}$ defined as $\mathcal{S}_1=\{k \in \{1,\ldots,n_d\}: r_k(\alpha_{i})=0, i \in \mathcal{A}\}$. If $\mathcal{S}_1$ is not empty we consider the set $\mathcal{S}_2 \subseteq \bar{\mathcal{A}}$ that contains all $i \in \bar{\mathcal{A}}$ such that $\sum_{j \in \mathcal{S}_1} r_j(\alpha_i)>0$. 
 
If $\mathcal{S}_2= \bar{\mathcal{A}}$ then $p=\sum_{i \in \mathcal{S}_1}\frac{1}{|\mathcal{S}_1|}r_i$ is a pmf of a $\mathcal{Z}$-compatible table with uniform margins. If $\mathcal{S}_1=\emptyset$ or $\mathcal{S}_2 \subset  \bar{\mathcal{A}}$ no $\mathcal{Z}$-compatible table with uniform margins exists.
\begin{proof}
We know that any $p$ associated with a table $T$ with uniform margins can be represented as $p=\sum_{k=1}^{n_d} \lambda_k r_k$ where $\lambda_k \geq 0$, $\sum_{k=1}^{n_d} \lambda_k=1$, and $r_k, k=1,\ldots,n_d$ are the extreme pmfs of the polytope defined by $H_d \,p=0$.
We must have
\[
p(\alpha_i)=\sum_{k=1}^{n_d} \lambda_k r_k(\alpha_i)=0, \; i \in \mathcal{A}.
\]
We can choose a strictly positive value for $\lambda_k$ only if the extreme pmf $r_k$ satisfies the condition $r_k(\alpha_i)=0, i \in \mathcal{A}$. If $\mathcal{S}_1$ is empty, no $\mathcal{Z}$-compatible pmf can be built.
If $\mathcal{S}_1$ is not empty, we have
\[
p=\sum_{k \in \mathcal{S}_1} \lambda_k r_k, \lambda_k \geq 0, \sum_{k \in \mathcal{S}_1} \lambda_k=1.
\]
The pmf $p$ must be positive for $\alpha_i$, $i \in \bar{\mathcal{A}}$:
\[
p(\alpha_i)=\sum_{k \in \mathcal{S}_1} \lambda_k r_k(\alpha_i)>0, \; i \in \bar{\mathcal{A}}.
\]
The coefficients $\lambda_k, k \in \mathcal{S}_1$ are positive or zero. Then we can consider $\sum_{k \in \mathcal{S}_1} r_k(\alpha_i)$. If $\mathcal{S}_2=\bar{\mathcal{A}}$ it would be enough to choose $\lambda_k=\frac{1}{|\bar{\mathcal{A}}|}$: the pmf $p=\sum_{k \in \bar{\mathcal{A}} } \frac{1}{|\bar{\mathcal{A}}|} r_k$ has uniform margins and is $\mathcal{Z}$-compatible by construction.  
If $\mathcal{S}_2$ is a proper subset of $\bar{\mathcal{A}}$, $\mathcal{S}_2 \subset \bar{\mathcal{A}}$ no $\mathcal{Z}$-compatible pmf can be built.
\end{proof}
\end{theorem}

\begin{corollary} \label{cor:morezeros}
Let $\mathcal{Z}$ be a zero pattern with $\mathcal{Z}=(z_i, i=1,\ldots,2^d)$. If $\mathcal{S}_1$ is empty then no $\mathcal{Z'}$-compatible table with uniform margins exists for any $\mathcal{Z}' \leq \mathcal{Z}$, where   $\mathcal{Z}' \leq \mathcal{Z}$ means $z_i'\leq z_i, \; i=1,\ldots, 2^d$, $\mathcal{Z}'=(z_1',\ldots,z_{2^d}')$.
\end{corollary}

\begin{example}
    
\noindent We consider some examples for $d=3$ with zero patterns reported in the last columns of Table~\ref{tab:ex3z}.
{\bf Case 1.} The first case, $\mathcal{Z}_1$, is a $7$-zero pattern. For this case, we get $\mathcal{S}_1=\emptyset$ and then no table with uniform margins exists for this case. 

\noindent{\bf Case 2.} The second case, $\mathcal{Z}_2$, is a $6$-zero pattern. Here, we get $\mathcal{S}_1=\emptyset$. Thus, in such a case, no table with uniform margins exists. Using Cor.~(\ref{cor:morezeros}) we can conclude that for any zero pattern $\mathcal{Z}'$, $z'=(z_1,z_2,0,0,0,0,0,0)$, $z_1,z_2 \in\rf{\{0,1\}}$ no $\mathcal{Z}'$-compatible table with uniform margins exists. 

\noindent {\bf Case 3.} The third case, $\mathcal{Z}_3$, is a $2$-zero pattern. We get $\mathcal{S}_1=\{2,4\}$. In this case, $\mathcal{A}=\{1,3\}$, $\bar{\mathcal{A}}=\{2,4,5,6,7,8\}$, and $\mathcal{S}_2=\{2,4,5,7\} \rf{\subset} \bar{\mathcal{A}}$. Then no table with uniform margins exists for this case. 

\noindent{\bf Case 4.} The fourth case, $\mathcal{Z}_4$, is a $4$-zero pattern. We obtain $\mathcal{S}_1=\{2,4\}$. Here, $\mathcal{A}=\{1,3,6,8\}$, $\bar{\mathcal{A}}=\{2,4,5,7\}$, and $\mathcal{S}_2=\{2,4,5,7\} = \bar{\mathcal{A}}$. Then $p=\sum_{k \in \{2,4\}} \frac{1}{2} r_k=(0\,0.25\,0\,0.25\,0.25\,0\, 0.25\,0)$ has uniform margin\rf{s} and is compatible with $\mathcal{Z}_4$.
\begin{table}[b!]
	\centering
		\begin{tabular}{ccc|rrrrrr|cccc}
$\alpha_1$ &	$\alpha_2$ &	$\alpha_3$ &	$r_1$ & 		$r_2$ & 		$r_3$ & 		$r_4$ & 		$r_5$ & 		$r_6$  & $\mathcal{Z}_1$ & $\mathcal{Z}_2$ & $\mathcal{Z}_3$ & $\mathcal{Z}_4$\\
		\hline
0 &	0 &	0 &	$\frac{1}{2}$ &	$0$ &	$0$ &	$0$ &	$\frac{1}{4}$&	$0$ & 1 & 1 & 0 & 0\\
1 &	0 &	0 &	$0$ &	$\frac{1}{2}$ &	$0$ &	$0$ &	$0$ & $\frac{1}{4}$ & 0 & 1 & 1 & 1\\
0 &	1 &	0 &	$0$ &	$0$  & $\frac{1}{2}$ &	$0$ &	$0$ & $\frac{1}{4}$ & 0 & 0 & 0 & 0\\
1 &	1 &	0 &	$0$ &	$0$  & $0$  & $\frac{1}{2}$ & $\frac{1}{4}$ &	$0$ & 0 & 0 & 1 & 1 \\
0 &	0 &	1 &	$0$ &	$0$  & $0$  & $\frac{1}{2}$  &	$0$ & $\frac{1}{4}$ & 0 & 0 & 1 & 1\\
1 &	0 &	1 &	$0$ &	$0$  & $\frac{1}{2}$ &	$0$ & $\frac{1}{4}$ & $0$  & 0 & 0 & 1 & 0\\
0 &	1 &	1 &	$0$ &	$\frac{1}{2}$ &	$0$ &	$0$ & $\frac{1}{4}$ &	$0$  & 0 & 0 & 1 & 1\\
1 &	1 &	1 &	$\frac{1}{2}$ &	$0$ &	$0$ &	$0$ &	$0$ & $\frac{1}{4}$ & 0 & 0 & 1 & 0\\
		\end{tabular}
	\caption{$d=3$: Extreme pmfs with uniform margins and four zero patterns.}
	\label{tab:ex3z}
\end{table}
\end{example}
As previously discussed, given the extreme pmfs \( r_i \), \( i=1,\ldots, n_d \), any pmf with uniform margins can be expressed as follows: 
\[
p = \sum_{i=1}^{n_d} \lambda_i r_i, \quad \lambda_i \geq 0, \quad \sum_{i=1}^{n_d} \lambda_i = 1.
\]
As a consequence, we have that each extreme pmf \( r_i \) itself represents a table with uniform margins that is compatible with \( \mathcal{Z} = (z_1, \ldots, z_{2^d}) \), where \( z_k = 0 \) (or \( z_k = 1 \)) if \( r_i(\alpha_k) = 0 \) (or \( r_i(\alpha_k) > 0 \), respectively), for \( k=1, \ldots, 2^d \) and \( i=1, \ldots, n_d \). Moreover, it is possible to sample from the set of tables with uniform margins by selecting \( \lambda_i \), \( i=1, \ldots, n_d \), according to some criterion within the \( (n_d-1) \)-dimensional simplex.

On the other hand, it is important to note that the cardinality \( n_d \) of the set of extreme pmfs grows rapidly with \( d \). For instance, when \( d=6 \), there are \( n_6 = 707,264 \) extreme pmfs, which poses a significant computational challenge for this approach.


An alternative way to tackle our problem would be through integer programming. 
Since our problem reduces to linear constraints through the Lemma~\ref{simpllemma}, we can set up a linear programming problem as follows. 
We start by considering the matrix $C_d$ with dimensions $(d+1)\times 2^d$ with: A first row of $1$'s; then block-wise repetitions of $1$'s and $-1$'s (as in a classical full-factorial $2^d$ design). 
$C_d$ is the matrix $H_d$ above except for the first row. 
Also, let us define the vector with length $d+1$ as $b_d=(1,0, \ldots, 0)$. 
For a given zero pattern ${\mathcal Z}$, we introduce the matrix $C_d(\mathcal{Z})$ obtained from $C_d$ by removing the columns that belong to the cells in $\mathcal{Z}$. 
The problem of the existence of a strictly positive probability distribution with uniform margins can be stated as follows:

\begin{equation} \label{eq:ip}
    \begin{array}{lcc}
     & \text{max}_p \delta & \\[3pt]
    {\text{subject to }} &    C_d(\mathcal{Z})p=b_d  & \\
    &  p_\alpha \geq \delta, &  \alpha \in \mathcal{D}-\mathcal{Z}. \\
\end{array}
\end{equation}

If the maximum is $\delta^*=0$, then there is no strictly positive solution. 
If the maximum $\delta^*$ is strictly positive, the solution exists and can be found through a standard algorithm based on the simplex method which yields a strictly positive probability table together with the value of $\delta^*$. \fr{We do not show the output table, because the linear programming problem in Eq.~\eqref{eq:ip} is only a compatibility check between the uniform margins and a given zero pattern.}
The algorithm above can be easily implemented with the \texttt{lpsolve} package, see \cite{lpSolve} and \cite{Buttrey2005}, in \texttt{R}\cite{Rcore}. 
An alternative version of the optimization problem can be stated using separation theorems like the Farkas lemma, see, e.g., \cite{Mangasarian1994}, but the refinement of the optimization problem is outside the scope of this work. 

The approach via \texttt{lpsolve} is quite efficient and allows finding a solution even for large binary tables. For example, the analysis for $d=12$ is performed in less than 1 sec.~on standard hardware. 
For illustrative purposes, we conclude with two examples in the case $d=4$. They are depicted in Fig.~\ref{fig4var}. In the upper panel, Fig.~\ref{fig4var} shows a table with 4 zeros on a two-dimensional sub-table as zero pattern. 
According to Prop.~\ref{prop:segmenti}, there is no strictly positive solution, and the linear programming algorithm yields a value of $\delta^*$=0.  
In the lower panel, Fig.~\ref{fig4var} shows a different zero pattern, which is a table with again 4 zeros but in different positions. In this case, there is a strictly positive solution, and the output of the optimization procedure is $\delta^*=0.07143$.

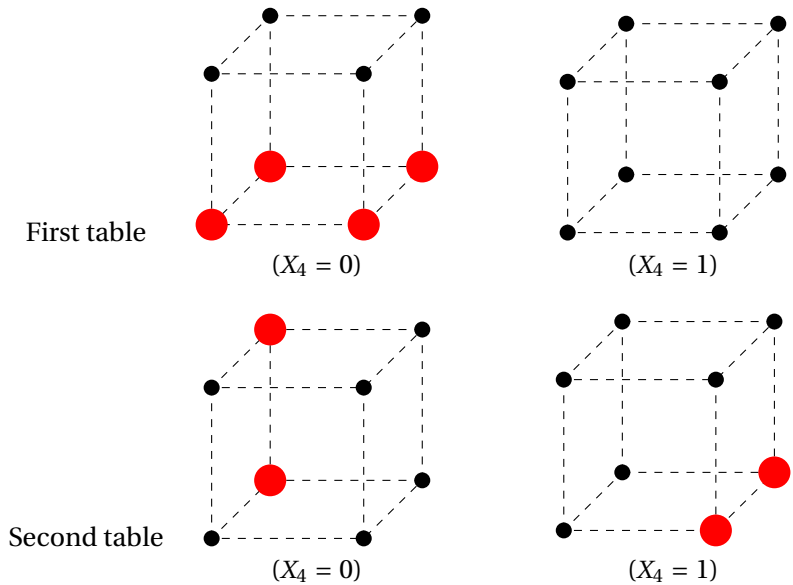
\begin{figure}
    \centering
    \begin{tabular}{cccc}
    First table &
        \begin{tikzpicture}[scale=2]
    \coordinate (1) at (0, 0, 0);
    \coordinate (2) at (1, 0, 0);
    \coordinate (3) at (0, 0, 1);
    \coordinate (4) at (1, 0, 1);
    \coordinate (5) at (0, 1, 0);
    \coordinate (6) at (1, 1, 0);
    \coordinate (7) at (0, 1, 1);
    \coordinate (8) at (1, 1, 1);

    \draw[dashed] (1) -- (2) -- (4) -- (3) -- cycle; 
    \draw[dashed] (5) -- (6) -- (8) -- (7) -- cycle; 
    \draw[dashed] (1) -- (5); 
    \draw[dashed] (2) -- (6);
    \draw[dashed] (3) -- (7);
    \draw[dashed] (4) -- (8);
    \foreach \point in {5, 6, 7, 8} {
        \fill[black] (\point) circle(1.5pt);
    }
    \foreach \point in {1,2,3,4} {
        \fill[red] (\point) circle(3pt);
    }
\end{tikzpicture}
& $ \qquad $ &
  \begin{tikzpicture}[scale=2]
    \coordinate (1) at (0, 0, 0);
    \coordinate (2) at (1, 0, 0);
    \coordinate (3) at (0, 0, 1);
    \coordinate (4) at (1, 0, 1);
    \coordinate (5) at (0, 1, 0);
    \coordinate (6) at (1, 1, 0);
    \coordinate (7) at (0, 1, 1);
    \coordinate (8) at (1, 1, 1);

    \draw[dashed] (1) -- (2) -- (4) -- (3) -- cycle; 
    \draw[dashed] (5) -- (6) -- (8) -- (7) -- cycle; 
    \draw[dashed] (1) -- (5); 
    \draw[dashed] (2) -- (6);
    \draw[dashed] (3) -- (7);
    \draw[dashed] (4) -- (8);
    \foreach \point in {1, 2, 3,4, 5, 6, 7, 8} {
        \fill[black] (\point) circle(1.5pt);
    }
    \foreach \point in {} {
        \fill[red] (\point) circle(3pt);
    }
\end{tikzpicture} \\
& ($X_4=0$) & & ($X_4=1$)  \\
 & & $ $ & \\
 Second table &
        \begin{tikzpicture}[scale=2]
    \coordinate (1) at (0, 0, 0);
    \coordinate (2) at (1, 0, 0);
    \coordinate (3) at (0, 0, 1);
    \coordinate (4) at (1, 0, 1);
    \coordinate (5) at (0, 1, 0);
    \coordinate (6) at (1, 1, 0);
    \coordinate (7) at (0, 1, 1);
    \coordinate (8) at (1, 1, 1);

    \draw[dashed] (1) -- (2) -- (4) -- (3) -- cycle; 
    \draw[dashed] (5) -- (6) -- (8) -- (7) -- cycle; 
    \draw[dashed] (1) -- (5); 
    \draw[dashed] (2) -- (6);
    \draw[dashed] (3) -- (7);
    \draw[dashed] (4) -- (8);
    \foreach \point in {2, 3, 4, 6, 7, 8} {
        \fill[black] (\point) circle(1.5pt);
    }
    \foreach \point in {1, 5} {
        \fill[red] (\point) circle(3pt);
    }
\end{tikzpicture}
& $ \qquad $ &
  \begin{tikzpicture}[scale=2]
    \coordinate (1) at (0, 0, 0);
    \coordinate (2) at (1, 0, 0);
    \coordinate (3) at (0, 0, 1);
    \coordinate (4) at (1, 0, 1);
    \coordinate (5) at (0, 1, 0);
    \coordinate (6) at (1, 1, 0);
    \coordinate (7) at (0, 1, 1);
    \coordinate (8) at (1, 1, 1);

    \draw[dashed] (1) -- (2) -- (4) -- (3) -- cycle; 
    \draw[dashed] (5) -- (6) -- (8) -- (7) -- cycle; 
    \draw[dashed] (1) -- (5); 
    \draw[dashed] (2) -- (6);
    \draw[dashed] (3) -- (7);
    \draw[dashed] (4) -- (8);
    \foreach \point in {1, 3, 5, 6, 7, 8} {
        \fill[black] (\point) circle(1.5pt);
    }
    \foreach \point in {2, 4} {
        \fill[red] (\point) circle(3pt);
    }
\end{tikzpicture} \\
& ($X_4=0$) & & ($X_4=1$)  \\
\end{tabular}
\caption{Two $2^4$ tables with zero patterns. The first table in the upper panel, the second table in the lower panel. The variable\rf{s} $X_1, X_2, X_3$ are the coordinates of the cubes, the variable $X_4$ divides the table into two cubes. Red dots denote zero probabilities. To ease readability, vertex labels are omitted.}
    \label{fig4var}
\end{figure}

In summary, a key advantage of the linear programming-based method is its computational efficiency. 
Conversely, the method based on extreme pmfs offers valuable insight into the structure of zero patterns that are compatible with the existence of tables with uniform margins.
The computation of extreme pmfs for the class of multivariate Bernoulli distributions with given margins, particularly in high dimensions, can also benefit from recent work (e.g., see \cite{fontana2024high}). 
We plan to further develop this method in the future by exploring this connection further.

\subsection{Examples in the $2^3$ case} \label{3wayex}

In this section, we show how some $2^3$ tables with only a few zeros can behave differently. 
Specifically, we give explicit configurations where \ep{the solution does not exist or it is not determined only by the conditional odds ratios}. 
For our goal, it is sufficient to analyze $2^3$ tables with one or two zeros, leading (up to isomorphism) to four possible configurations. 
We recall that, in our context, two tables are isomorphic if one can be obtained from the other by permuting the levels and relabeling the variables.

\begin{enumerate}
\item The table in Fig.~\ref{fig:conf}~(a) has one zero. In this configuration, there are still four independent equations from marginal constraints and three independent \ep{conditional} odd ratios equations from the system in Eq.~\eqref{orsys}. Thus, the number of equations is equal to the number of positive probabilities.

\item The table in Fig.~\ref{fig:conf}~(b) has two zero on the same edge, specifically in the entries $(0,1,0)$ and $(1,1,0)$. According to Prop.~\ref{prop:segmenti}, there is no solution unless we also force the two opposite cells in $(0,0,1)$ and $(1,0,1)$ to be zero.

\item The table in Fig.~\ref{fig:conf}~(c) has two zeros on the opposite cells of a face. In this configuration, there are four independent equations from marginal constraints and one independent conditional odds ratios equation from the system in Eq.~\eqref{orsys}. \ep{Since the positive probabilities are six, in this case the conditional odds-ratios in Eq.~\eqref{orsys} are not sufficient to determine the transformed table and higher order combinations of odds ratios should be considered.}

\item The table in Fig.~\ref{fig:conf}~(d) has two zeros on two different faces. \fr{Here, there are no valid constraints on the conditional odds ratios in Eq.~\eqref{orsys}, and the number of positive probabilities is two more than the number of marginal constraints.}
\end{enumerate}

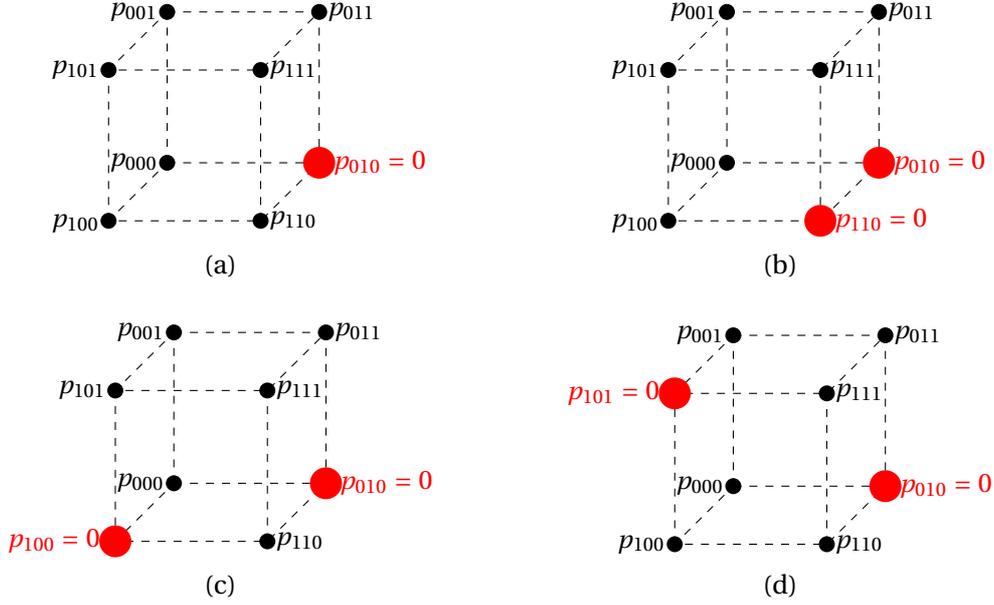
\begin{figure}
    \centering
    \begin{tabular}{ccc}
        \begin{tikzpicture}[scale=2]
    \coordinate (1) at (0, 0, 0);
    \coordinate (2) at (1, 0, 0);
    \coordinate (3) at (0, 0, 1);
    \coordinate (4) at (1, 0, 1);
    \coordinate (5) at (0, 1, 0);
    \coordinate (6) at (1, 1, 0);
    \coordinate (7) at (0, 1, 1);
    \coordinate (8) at (1, 1, 1);

    \draw[dashed] (1) -- (2) -- (4) -- (3) -- cycle; 
    \draw[dashed] (5) -- (6) -- (8) -- (7) -- cycle; 
    \draw[dashed] (1) -- (5); 
    \draw[dashed] (2) -- (6);
    \draw[dashed] (3) -- (7);
    \draw[dashed] (4) -- (8);
    \foreach \point in {1, 3, 4, 5, 6, 7, 8} {
        \fill[black] (\point) circle(1.5pt);
    }
    \foreach \point in {2} {
        \fill[red] (\point) circle(3pt);
    }

    \node[left] at (1) {$p_{000}$};
    \node[right] at (2) {$\ \color{red}{p_{010}=0}$};
    \node[left] at (3) {$\phantom{0=}p_{100}$};
    \node[right] at (4) {$p_{110}$};
    \node[left] at (5) {$p_{001}$};
    \node[right] at (6) {$p_{011}$};
    \node[left] at (7) {$p_{101}$};
    \node[right] at (8) {$p_{111}$};
\end{tikzpicture}
& $ \qquad $ &
  \begin{tikzpicture}[scale=2]
    \coordinate (1) at (0, 0, 0);
    \coordinate (2) at (1, 0, 0);
    \coordinate (3) at (0, 0, 1);
    \coordinate (4) at (1, 0, 1);
    \coordinate (5) at (0, 1, 0);
    \coordinate (6) at (1, 1, 0);
    \coordinate (7) at (0, 1, 1);
    \coordinate (8) at (1, 1, 1);

    \draw[dashed] (1) -- (2) -- (4) -- (3) -- cycle; 
    \draw[dashed] (5) -- (6) -- (8) -- (7) -- cycle; 
    \draw[dashed] (1) -- (5); 
    \draw[dashed] (2) -- (6);
    \draw[dashed] (3) -- (7);
    \draw[dashed] (4) -- (8);
    \foreach \point in {1, 3, 5, 6, 7, 8} {
        \fill[black] (\point) circle(1.5pt);
    }
    \foreach \point in {2, 4} {
        \fill[red] (\point) circle(3pt);
    }

    \node[left] at (1) {$p_{000}$};
    \node[right] at (2) {$\ \color{red}{p_{010}=0}$};
    \node[left] at (3) {$\phantom{0=}p_{100}$};
    \node[right] at (4) {$\ \color{red}{p_{110}=0}$};
    \node[left] at (5) {$p_{001}$};
    \node[right] at (6) {$p_{011}$};
    \node[left] at (7) {$p_{101}$};
    \node[right] at (8) {$p_{111}$};
\end{tikzpicture} \\
(a) & & (b)  \\
 &  $ $ & \\
        \begin{tikzpicture}[scale=2]
    \coordinate (1) at (0, 0, 0);
    \coordinate (2) at (1, 0, 0);
    \coordinate (3) at (0, 0, 1);
    \coordinate (4) at (1, 0, 1);
    \coordinate (5) at (0, 1, 0);
    \coordinate (6) at (1, 1, 0);
    \coordinate (7) at (0, 1, 1);
    \coordinate (8) at (1, 1, 1);

    \draw[dashed] (1) -- (2) -- (4) -- (3) -- cycle; 
    \draw[dashed] (5) -- (6) -- (8) -- (7) -- cycle; 
    \draw[dashed] (1) -- (5); 
    \draw[dashed] (2) -- (6);
    \draw[dashed] (3) -- (7);
    \draw[dashed] (4) -- (8);
    Aggiunta dei punti
    \foreach \point in {1, 4, 5, 6, 7, 8} {
        \fill[black] (\point) circle(1.5pt);
    }
    \foreach \point in {2, 3} {
        \fill[red] (\point) circle(3pt);
    }

    \node[left] at (1) {$p_{000}$};
    \node[right] at (2) {$\ \color{red}{p_{010}=0}$};
    \node[left] at (3) {$\color{red}{p_{100}=0} \ $};
    \node[right] at (4) {$p_{110}$};
    \node[left] at (5) {$p_{001}$};
    \node[right] at (6) {$p_{011}$};
    \node[left] at (7) {$p_{101}$};
    \node[right] at (8) {$p_{111}$};
\end{tikzpicture}
& $ \qquad $ &
  \begin{tikzpicture}[scale=2]
    \coordinate (1) at (0, 0, 0);
    \coordinate (2) at (1, 0, 0);
    \coordinate (3) at (0, 0, 1);
    \coordinate (4) at (1, 0, 1);
    \coordinate (5) at (0, 1, 0);
    \coordinate (6) at (1, 1, 0);
    \coordinate (7) at (0, 1, 1);
    \coordinate (8) at (1, 1, 1);

    \draw[dashed] (1) -- (2) -- (4) -- (3) -- cycle; 
    \draw[dashed] (5) -- (6) -- (8) -- (7) -- cycle; 
    \draw[dashed] (1) -- (5); 
    \draw[dashed] (2) -- (6);
    \draw[dashed] (3) -- (7);
    \draw[dashed] (4) -- (8);
    \foreach \point in {1, 3, 4, 5, 6, 8} {
        \fill[black] (\point) circle(1.5pt);
    }
    \foreach \point in {2, 7} {
        \fill[red] (\point) circle(3pt);
    }

    \node[left] at (1) {$p_{000} $};
    \node[right] at (2) {$\ \color{red}{p_{010}=0}$};
    \node[left] at (3) {$p_{100}$};
    \node[right] at (4) {$p_{110}$};
    \node[left] at (5) {$p_{001}$};
    \node[right] at (6) {$p_{011}$};
    \node[left] at (7) {$\color{red}{p_{101}=0} \ $};
    \node[right] at (8) {$p_{111}$};
\end{tikzpicture} \\
(c) & & (d) 
\end{tabular}
\caption{Configurations with 1 or 2 zeros in the $2^3$ table. Red dots denote zero probabilities. (a): one zero; (b): two zeros on the same edge; (c): two zeros in diagonal position on a face; (d): two zeros in different faces.}
    \label{fig:conf}
\end{figure}

\fr{A discussion on cases (c) and (d) above is provided with full detail in Sect.~\ref{Sec:zero-uniqueness}.}

\subsection{\rf{Extension of the approach to arbitrary $d$-way tables}} \label{dway}

\rf{It is worth noting that both the computational methods described in Sect.~\ref{sec:dwaytab} can be generalized to any $d$-way table, where each variable $X_i$ takes values in $\{0,1,\ldots,x_i-1\}$ for $x_i \geq 2$, with $i=1,\ldots,d$. 
For this, it is sufficient to modify $H_d$ and $C_d$ to express the constraints
\[
m_{i,0}=\ldots=m_{i,x_i-1}=\frac{1}{x_i},  \; i=1\ldots,d\, ,
\]
where $m_{i,j}=\sum_{\alpha \in \mathcal{D},\alpha_i=j} p_\alpha$, $i=1,\ldots,d, \; j=0,\ldots,x_i-1$, and $\mathcal{D}=\{0,\ldots,x_1-1\} \times \ldots \{0,\ldots,x_d-1\}$. We get
\[
m_{i,0}=\ldots=m_{i,x_i-1}=\frac{1}{x_i} \Leftrightarrow 
\begin{cases}
    m_{i,0}-m_{i,1}=0 \\
    \ldots \\
    m_{i,x_i-2}-m_{i,x_i-1}=0
\end{cases}, \; i=1,\ldots,d
\]
that can be rewritten as follows
\begin{equation} \label{eq:systemdway}
 (1_{\{\alpha \in \mathcal{D}, \alpha_i=j\}}-1_{\{\alpha \in \mathcal{D}, \alpha_i=j+1\}})^Tp=0, \;\;i=1\ldots,d, \,  j=0,\ldots,x_i-2.
\end{equation}
For example, in the case of a $3 \times 4$ table ($d=2, x_1=3$, and $x_2=4$) the $H_d$ matrix of Eq.(\ref{eq:HD}) becomes 
\[
H_{x_1,x_2}=
\begin{pmatrix}
1 & 1 & 1 & 1 & -1 & -1 & -1 & -1 & 0 & 0 & 0 & 0 \\
0 & 0 & 0 & 0 & 1 & 1 & 1 & 1 & -1 & -1 & -1 & -1 \\
1 & -1 & 0 & 0 & 1 & -1 & 0 & 0 & 1 & -1 & 0 & 0 \\
0 & 1 & -1 & 0 & 0 & 1 & -1 & 0 & 0 & 1 & -1 & 0 \\
0 & 0 & 1 & -1 & 0 & 0 & 1 & -1 & 0 & 0 & 1 & -1 
\end{pmatrix}
\]
}
\rf{Given a zero pattern $\mathcal{Z}$, we have two ways to check if a $\mathcal{Z}$ compatible table with uniform margins exists:
\begin{enumerate}
    \item using 4ti2 \cite{4ti2}, we compute the rays corresponding to the matrix $H_{x_1,x_2}$. We obtain $96$ extreme pmfs and then we proceed in a similar way as in Thm.~\ref{thm:extremerays};  
    \item we build the $C_{x_1,x_2}$ matrix, by simply adding a first row of 1's to $H_{x_1,x_2}$ and then we use the integer programming approach, as in \ref{eq:ip}. 
\end{enumerate}
\fr{We point out that the methodology described here formally solves the problem of existence for general $d$-way tables. However, in the multi-dimensional setting the geometric structure of the zero-patterns becomes complicated and a straightforward generalization of Thm.~\ref{th:Geenens_6.1} is hard to obtain if not unfeasible.}
}

\section{\ep{The role of conditional odds ratios in determining the unique transformed table}}
\label{Sec:zero-uniqueness}
\ep{In the previous section, we focused on the existence of the solution. We first investigate in which cases the set of conditional odds ratios is sufficient to describe the IPFP solution without the addition of higher order combinations of odds ratios. We start with conditional odds ratios due to their simple interpretation in terms of local dependence. Our analysis goes along the same line as the two-way table example presented in Sec.~\ref{Sec:background}. Here, we analyze two scenarios in the three-dimensional case, to show that more complex equations play a role to guarantee uniqueness.}

\ep{
\begin{table}[b]
	\centering
		\begin{tabular}{c|rrrrrrrr|rrr}
pmf & $000$ & $001$ & $010$ & $011$ & $100$ & $101$ & $110$ & $111$ & $\omega_M^{12}$ &
$\omega_M^{13}$ & $\omega_M^{23}$\\
\hline
$p_0$ & $0.1$ & $0.05$ & $0.3$ & $0.2$ & $0.1$ & $0.05$ & $0.15$ & $0.05$ & $0.4$ & $0.64$ & $1.111$ \\  
$p_1$ & $0.09$ & $0.09$ & $0.14$ & $0.18$ & $0.16$ & $0.16$ & $0.11$ & $0.07$ & $0.357$ & $0.714$ &  $1.033$ \\
\end{tabular}
	\caption{\rf{Marginal odds ratios of $p_0$ (input) and $p_1$ (output obtained using IPFP), all values rounded to 3 decimal places}}
	\label{tab:marg01}
\end{table}
First, it is worth noting that the odds ratios considered are the \emph{conditional} odds ratios and not the \emph{marginal} odds ratios, which are in general not preserved by the IPFP algorithm. We show this phenomenon in Table \ref{tab:marg01} through a specific example. There, the marginal odds ratios (denoted by $\omega_M^{12}$, $\omega_M^{13}$, and $\omega_M^{23}$) of $p_0$ (the input pmf) and $p_1$ (the output pmf obtained using IPFP) are different, meaning that, in general, the marginal odds ratios are not preserved in the transformation.
%
}

\rf{Let us consider a pmf $p_0$ with a zero pattern compatible with the existence of a table with uniform margins. Given $p_0$ as input, IPFP produces as output $p_1$, a pmf with uniform margins that has the same odds ratio structure of $p_0$. As previously discussed, it is known that $p_1$ is necessarily the I-projection of $p_0$ onto the class of $d$-variate Bernoulli distributions with uniform margins, and that such a I-projection is unique (Theorem 6.2 in \cite{geenens2020copula} or Proposition 2.1 in \cite{kojadinovic2024}).  
In this section, we consider $3$-dimensional binary tables and highlight the roles of the conditional odds ratios in determining the I-projection.
}

\begin{example} \label{ex:2zeros}
\rf{We consider a $2^3$ table with zeros in the entries $(0,1,0)$ and $(1,0,0)$, that is $p_{010}=p_{100}=0$ (Fig.~\ref{fig:conf}(c)). It can be verified that this zero-pattern is compatible with the existence of a table with uniform margins.  
In this case, we have only one non-trivial condition on the conditional odds ratio, precisely $\omega_{1}^{12}$, i.e. the odds ratio of $X_1,X_2$ given $X_3=1$: 
\begin{equation} \label{onlyone}
\omega_{1}^{12}=\frac {p_{001}p_{111}} {{p_{011}p_{101}}}.
\end{equation}
Additionally, we have four independent conditions for the margins, $C_3(\mathcal{Z})p=b_3$ as defined in Sec.~\ref{sec:dwaytab}.
Again using the lexicographic order for the entries, the kernel of the model matrix $C_3(\mathcal{Z})$ is generated by the two vectors:
\[
(-1,1,0,0,1,-1), \qquad (0,1,-1,-1,0,1).
\]
Given a particular solution $\tilde p$, for instance, the solution of the IPFP which is strictly positive, the general solution to the problem with fixed margins is given by
\begin{multline} \label{gendistr}
\left(
p_{000}=\tilde{p}_{000}-\alpha, p_{001}=\tilde{p}_{001}+\alpha+\beta, p_{011}=\tilde{p}_{011}-\beta, \right. \\
\left. p_{101}=\tilde{p}_{101}-\beta , p_{110}=\tilde{p}_{110}+\alpha, p_{111}=\tilde{p}_{111}-\alpha+\beta
\right)   
\end{multline}
for suitable values of $(\alpha,\beta)$, i.e. $\{(\alpha,\beta): p_{ijk} \geq 0, i,j,k \in \{0,1\}\}$. This distribution is shown in Fig.~\ref{fig:nonuni}.
By substitution of the generic probability distribution in Eq.~\eqref{gendistr} into the odds ratio equation in Eq.~\eqref{onlyone}, we get:
\begin{equation}\label{alphabeta}
f(\alpha,\beta)=(\tilde{p}_{111}-\tilde{p}_{001})\alpha +(\tilde{p}_{111}+\tilde{p}_{001}+\omega_{1}^{12}(\tilde{p}_{011}+\tilde{p}_{101}))\beta -\alpha^2+(1-\omega_{1}^{12})\beta^2=0\, .
\end{equation}}
\begin{figure}[b!]
    \centering
\begin{tikzpicture}[scale=2.2]
    \coordinate (1) at (0, 0, 0);
    \coordinate (2) at (1, 0, 0);
    \coordinate (3) at (0, 0, 1);
    \coordinate (4) at (1, 0, 1);
    \coordinate (5) at (0, 1, 0);
    \coordinate (6) at (1, 1, 0);
    \coordinate (7) at (0, 1, 1);
    \coordinate (8) at (1, 1, 1);

    \draw[dashed] (1) -- (2) -- (4) -- (3) -- cycle; 
    \draw[dashed] (5) -- (6) -- (8) -- (7) -- cycle; 
    \draw[dashed] (1) -- (5); 
    \draw[dashed] (2) -- (6);
    \draw[dashed] (3) -- (7);
    \draw[dashed] (4) -- (8);
    \foreach \point in {1, 4, 5, 6, 7, 8} {
        \fill[black] (\point) circle(1pt);
    }
    \foreach \point in {2, 3} {
        \fill[red] (\point) circle(2pt);
    }

    \node[left] at (1) {$\tilde{p}_{000}+\alpha\phantom{x} $};
    \node[right] at (2) {$\phantom{x} \color{red}{\tilde{p}_{010}=0}$};
    \node[left] at (3) {$\color{red}{\tilde{p}_{100}=0\phantom{x}} \ $};
    \node[right] at (4) {$\phantom{x}\tilde{p}_{110}-\alpha$\ };
    \node[left] at (5) {$\tilde{p}_{001}-\alpha+\beta\phantom{x}$\ };
    \node[right] at (6) {$\phantom{x}\tilde{p}_{011}-\beta$\ };
    \node[left] at (7) {$\tilde{p}_{101}-\beta\phantom{x}$\ };
    \node[right] at (8) {$\phantom{x}\tilde{p}_{111}+\alpha+\beta$\ };
\end{tikzpicture}
\caption{The general probability table in the $2^3$ case with two zeros in the diagonal of a face.}
    \label{fig:nonuni}
\end{figure}
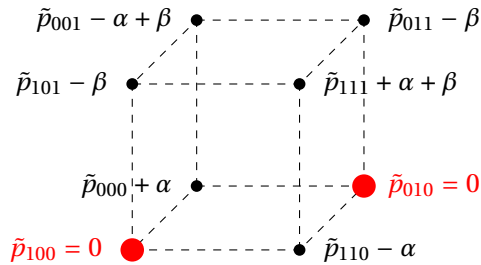
\rf{The equation above introduces an implicit function $\beta(\alpha)$, whose explicit expression is somewhat complicated and not relevant for our application, having quadratic terms in both $\alpha$ and $\beta$. 
However, we can see that, at least locally, we have infinite solutions in a neighborhood of the point $(0,0)$.
The point $(0,0)$ is a solution because $\tilde{p}$ is the probability distribution given by the IPFP. Since $\tilde{p}$ is strictly positive, there is a neighborhood of $(0,0)$ in the strictly positive simplex of the probabilities. The local behavior of $\beta(\alpha)$ can be studied by means of the implicit functions theorem. The partial derivatives are
\[
\frac{\partial f}{\partial \alpha} = \tilde{p}_{111}-\tilde{p}_{001}-2\alpha, \qquad \frac{\partial f}{\partial \beta} = \tilde{p}_{111}+\tilde{p}_{001}+\omega_{1}^{12}(\tilde{p}_{011}+\tilde{p}_{101})+2(1-\omega_{1}^{12})\beta \, ,
\]
leading to
\[
\beta(0)=0, \qquad \beta'(0) = - \frac {\tilde{p}_{111}-\tilde{p}_{001}} {\tilde{p}_{111}+\tilde{p}_{001}+\omega_{1}^{12}(\tilde{p}_{011}+\tilde{p}_{101})} \, .
\]
A special case arises for conditional independence, i.e., $\omega_{1}^{12}=1$ where the function in Eq.~\eqref{alphabeta} reduces to
\[
f(\alpha,\beta) = (\tilde{p}_{111}-\tilde{p}_{001})\alpha +\frac 1 2 \beta - \alpha^2 = 0
\]
from which
\[
\beta(\alpha) = 2\alpha(\alpha-(\tilde{p}_{111}-\tilde{p}_{001})) \, .
\]
We consider a numerical example, shown in Table \ref{tab:multiodd}, where we report the original $p_0$ and the corresponding $p_1$ found using IPFP. 
Then, we determine one solution $(\alpha_\star,\beta_\star)$ of Eq.~\eqref{alphabeta} and build $p_1'$ according to Eq.~\eqref{gendistr}. We can verify that $p_1'$ has uniform margin and has the same value of the conditional odds-ratio $\omega_{1}^{12}$ of $p_0$ and $p_1$. 
In this case, to understand why the solution is unique we need to compute combinations of odds ratios symbolically defined as the product of two conditional odds-ratios. More specifically, we consider the symbolic product of $\omega^{13}_0$ and $\omega^{23}_1$ that we denote by $\omega_{01}^{13,23}$ and obtain  
\[
\omega_{01}^{13,23}=\frac{p_{000}p_{111}}{p_{001}p_{110}}
\]
Looking at Table \ref{tab:multiodd}, we observe that $\omega_{01}^{13,23}$ of $p_1'$ is different to those of $p_0$ and $p_1$. It can also be verified that by adding the condition 
\[
\omega_{01}^{13,23}(\alpha,\beta)=\omega_{01}^{13,23}(p_0) \Leftrightarrow \omega_{01}^{13,23}(\alpha,\beta)=1
\]
to Eq.~\eqref{alphabeta}, we obtain $\alpha=\beta=0$ as the unique solution, which means that $p_1'$ must be chosen equal to $p_1$.
}

\begin{table}[t]
	\centering
		\begin{tabular}{c|rrrrrrrr|rr}
pmf & $000$ & $001$ & $010$ & $011$ & $100$ & $101$ & $110$ & $111$ & $\omega_{1}^{12}$ &
$\omega_{01}^{13,23}$ \\
\hline
$p_0$ & $0.288$ &  $0.106$ & $0$ & $0.106$ & $0$ & $0.106$ & $0.288$ & $0.106$ & $1$ & $1$ \\
$p_1$ & $0.250$ & $0.125$ & $0$ & $0.125$ & $0$ & $0.125$ & $0.250$ & $0.125$ & $1$ & $1$ \\ 
$p_1'$ & $0.240$ & $0.135$ & $0$ & $0.125$ & $0$ & $0.125$ & $0.260$ & $0.115$ & $1$ & $0.787$ \\ 
\end{tabular}
	\caption{Conditional odds-ratios of $p_0$ (input), $p_1$ (output obtained using IPFP), and $p'_1$ (all values rounded to 3 decimal places)}
	\label{tab:multiodd}
\end{table}
\end{example}

\begin{example} \label{ex:3zeros}
\rf{We consider the zero-pattern $\{(0,1,1),(1,0,1),(1,1,0)\}$, that is $p_{011}=p_{101}=p_{110}=0$. It can be verified that in this case, the polytope of the pmf with uniform margins is 
\begin{equation} \label{eq:lambda}
\{p=\lambda r_1 + (1-\lambda) r_6: 0 \leq \lambda \leq 1 \}    
\end{equation}
where $r_1$ and $r_6$ are the extreme pmfs defined in Table \ref{tab:ex3}. We consider an initial $p_0$ and find $p_1$ as the output of IPFP. Then, taking $\lambda=.9$ for illustrative purposes, we get $p'_1=0.9 r_1+0.1 r_6$ (see Table~\ref{tab:odd3m}).}
\begin{table}[b]
	\centering
		\begin{tabular}{c|rrrrrrrr|r}
pmf & $000$ & $001$ & $010$ & $011$ & $100$ & $101$ & $110$ & $111$ & $\omega^{12,13,23}_{001}$ \\
\hline
$p_0$ &  $0.40$ & $0.15$ & $0.15$ & $0$ & $0.15$ & $0$ & $0$ & $0.15$ & $7.111$\\
$p_1$ &  $0.225$ & $0.137$ & $0.137$ & $0$ & $0.137$ & $0$ & $0$ & $0.363$ & $7.111$\\
$p'_1$ &  $0.45$ & $0.025$ & $0.025$ & $0$ & $0.025$ & $0$ & $0$ & $0.475$ & $6156$\\ 
\end{tabular}
	\caption{\rf{$\omega^{12,13,23}_{001}$ odds-ratio of $p_0$ (input), $p_1$ (output obtained using IPFP), and $p'_1$ (all the values are rounded to 3 decimal places)}}
	\label{tab:odd3m}
\end{table}
\rf{It is possible to verify that, in this case, all conditional odds-ratios and all the products $\omega_{kk'}^{ij,i'j'}$, which are symbolically defined as the product $\omega^{ij}_k \omega^{i'j'}_{k'}$ are trivial. In this case, to retrieve the unique solution, we need to consider three-term products of conditional odds ratios like the following
\begin{equation} \label{e3or3}
\omega^{12,13,23}_{001}=\frac{ p_{000}^2 p_{111} }{ p_{010} p_{100} p_{001}}    
\end{equation}
where $\omega^{12,13,23}_{001}$ is the symbolic product of the conditional odds-ratios $\omega^{12}_0$, $\omega^{13}_0$, and $\omega^{23}_1$. Considering Table \ref{tab:odd3m}, we observe that $\omega^{12,13,23}_{001}$ has the same value for $p_0$ and $p_1$ but not for $p_1'$. In particular using Eq.~\eqref{eq:lambda} in Eq.~\eqref{e3or3} we obtain
\begin{equation} \label{eq:wlambda}
 \omega^{12,13,23}_{001}=\frac{ \lambda^2 (1+\lambda) }{ (1-\lambda)^3}    
\end{equation}
It is easy to verify that Eq.~\eqref{eq:wlambda} is invertible for $0 \leq \lambda < 1$ and this provides a further justification of the uniqueness of $p_1$ for this zero-pattern.}
\end{example}

\rf{We conclude this section by considering all possible $2^{2^d}=256$ zero patterns (of course the $8$-zero-pattern that results in a trivial all-zero table is excluded by the analysis). Using Thm.~\ref{thm:extremerays} and the extreme points listed in Table~\ref{tab:ex3z}, we find $45$ zero patterns compatible with the existence of a table with uniform margins. For each compatible zero pattern $\mathcal{Z}=(z_1,\ldots,z_8)$ we compute the extreme rays of the polytope $\mathcal{P}_{\mathcal{Z}}$ defined by $H_3p=0$ and $p_{\alpha_i}=0$ if $z_i=0$, $i=1,\ldots,8$. We denote by $N_0$ and $N_1$ the number of zeros in the zero pattern $\mathcal{Z}$ and the number of extreme points of the polytope $\mathcal{P}_{\mathcal{Z}}$, respectively.} 

\rf{Table \ref{tab:n0en1} reports the classification of the $45$ compatible zero patterns by $N_0$ and $N_1$. There are $6$ cases where $N_1=1$. For each of these zero patterns, there is only one table with uniform margins. For the remaining $39$ cases we compute the odds ratios, obtaining:
\begin{enumerate}
    \item for the zero pattern with $N_0 \in \{0,1,2,4\}$ zeros, the uniquess of the solution is determined by the non-trivial odds ratios which are conditional odds ratios and/or products of two conditional odds ratios, like in Example \ref{ex:2zeros};
    \item for the $3$-zero patterns the uniquess of the solution is determined by non-trivial odds ratios which are the products of three conditional odds ratios, like in Example \ref{ex:3zeros}.
\end{enumerate}
}
\begin{table}[t]
	\centering
\begin{tabular}{c|rrrrr|r}
& \multicolumn{5}{c|}{$N_1$} & \\
$N_0$      & 1 & 2  & 3  & 4 & 6 & Total \\
\hline
0     & 0 & 0  & 0  & 0 & 1 & 1     \\
1     & 0 & 0  & 0  & 8 & 0 & 8     \\
2     & 0 & 0  & 16 & 0 & 0 & 16    \\
3     & 0 & 8  & 0  & 0 & 0 & 8     \\
4     & 2 & 6  & 0  & 0 & 0 & 8     \\
6     & 4 & 0  & 0  & 0 & 0 & 4     \\
\hline
Total & 6 & 14 & 16 & 8 & 1 & 45   
\end{tabular} 
\caption{\rf{Cross classification of tables with uniform margins by $N_0$ and $N_1$, $d=3$.}}
\label{tab:n0en1}
\end{table}

\rf{The analysis of $3$-dimensional tables show how conditional odds ratios and their products play a role in determining the IPFP solution when it exists.}

\section{Conclusions}
\label{Sec:discussion}
In this paper, we \ep{provide a characterization of the zero-patterns that are compatible with a transformed multi-way binary table with prescribed dependence structure in terms of odds ratios. 
We further study how the zero-pattern impacts the relevance of conditional odds ratios in determining the unique solution of the problem.}
While working on this project, we identified interesting research directions for future investigations. 
%
\ep{For binary tables, the approach based on extreme rays described in Sec.~\ref{Sec:zero-existence} would benefit from a full characterization of the extreme rays to avoid the computational step.}
Some advances in this direction have recently been made in the context of multivariate Bernoulli distributions \cite{fontana2024high}. The possible adaptation of these results in the framework of polytopes of discrete copulas and higher order transportation polytopes is worth investigating. 
With these new insights, it would be possible to prove if the condition on the opposite facets of Prop.~\ref{prop:segmenti} could be generalized to multi-way tables with an arbitrary number of levels.
Moreover, it would be interesting to characterize the class solutions when \ep{conditional odds ratios are not enough to determine the transformed table } in terms of the associated algebraic variety. 
An algebraic approach would make it possible to identify extreme tables in this variety and study if there are valuable alternatives with respect to the solution provided by the IPFP. 
Finally, we could investigate how the approach adapts to different dependence conditions. 
For instance, we might consider suitable marginal odds ratios to have a closer connection with the classical log-linear models of independence and conditional independence for multi-way contingency tables.

\smallskip
\section*{Acknowledgments}
We are grateful to two anonymous reviewers for their comments which led to a substantial improvement of the article. Elisa Perrone thanks Jane Coons, Kaie Kubjas, Dorota Kurowicka, and Steffen Lauritzen for fruitful discussions on this work.

\bibliographystyle{elsarticle-num} 
\bibliography{references.bib}

\end{document}